\documentclass[11pt,a4paper,twoside]{amsart}
\usepackage[activeacute,english]{babel}
\usepackage[latin1]{inputenc}
\usepackage{amsfonts}
\usepackage{amsmath,amsthm}
\usepackage{amssymb}
\usepackage{graphics}

\newcommand{\CC}{{\mathcal C}}

\newcommand{\FF}{{\mathcal F}}

\newcommand{\XX}{{\mathcal X}}

\newcommand{\C}{{\mathbb C}}
\newcommand{\N}{{\mathbb N}}
\newcommand{\R}{{\mathbb R}}

\newcommand{\supp}{{\operatorname{supp}}}

\newcommand{\Tr}{{\operatorname{t}}}
\newcommand{\Ran}{{\operatorname{rn}}}
\newcommand{\Ker}{{\operatorname{kr}}}

\newtheorem{teo}{Theorem}[section]

\newtheorem{lema}[teo]{Lemma}
\newtheorem{coro}[teo]{Corollary}
\newtheorem{propo}[teo]{Proposition}

\newtheorem{ques}[teo]{Question}

{\theoremstyle{remark} \newtheorem{remark}[teo]{Remark}}

\title[Shell interactions for Dirac operators: point spectrum and confinement]{Shell interactions for Dirac operators:\\on the point spectrum and the confinement}
\author[N. Arrizabalaga, A. Mas, L. Vega]{Naiara Arrizabalaga, Albert Mas, Luis Vega}
\date{July, 2014}
\thanks{Arrizabalaga was supported in part by MTM2011-24054 and IT641-13. Mas was partially supported by MTM2010-16232, 2009SGR-000420, IT-641-13, and ERC-320501 (ANGEOM project).  Vega was partially supported by MTM2011-24054, UFI11/52 and IT641-13.}
\address{N. Arrizabalaga and A. Mas. Departamento de Matem\'aticas, Universidad del Pa\'is Vasco/Euskal Herriko Unibertsitatea, 48080 Bilbao (Spain)}
\email{naiara.arrizabalaga@ehu.es, amasblesa@gmail.com}
\address{L. Vega. Departamento de Matem\'aticas, Universidad del Pa\'is Vasco/Euskal Herriko Unibertsitatea, 48080 Bilbao and BCAM, Alameda de Mazarredo, 14
E-48009 Bilbao (Spain)} 
\email{luis.vega@ehu.es, lvega@bcamath.org}

\begin{document}

\begin{abstract}
Spectral properties and the confinement phenomenon for the coupling $H+V$ are studied, where $H=-i\alpha\cdot\nabla
+m\beta$ is the free Dirac operator in $\R^3$ and $V$ is a measure-valued potential. The potentials $V$ under consideration are given in terms of surface measures on the boundary of bounded regular domains in $\R^3$.

A criterion for the existence of point spectrum is given, with applications to electrostatic shell potentials.
In the case of the sphere, an uncertainty principle is developed and its relation with some eigenvectors of the coupling is shown.

Furthermore, a criterion for generating confinement is given. As an application, some known results about confinement on the sphere for electrostatic plus Lorentz scalar shell potentials are generalized to regular surfaces. 
\end{abstract}
\maketitle

\section{Introduction}

In this article we study the coupling of the free Dirac operator with measure-valued potentials developed in \cite{AMV}. The aim of the present work is to give spectral properties of these couplings and to investigate the phenomenon of confinement.
Given $m\geq0$, the free Dirac operator in $\R^3$ is defined by 
$H=-i\alpha\cdot\nabla+m\beta,$
where $\alpha=(\alpha_1,\alpha_2,\alpha_3)$,
\begin{equation}\label{freedirac}
\begin{split}
\alpha_j
=&\left(\begin{array}{cc} 0 & \sigma_j\\
\sigma_j & 0 \end{array}\right) 
\quad\text{for }j=1,2,3,\quad
\beta
=\left(\begin{array}{cc} I_2 & 0\\
0 & -I_2 \end{array}\right),\quad
I_2
=\left(\begin{array}{cc} 1 & 0\\
0 & 1 \end{array}\right),\\
&\text{and}\quad
\sigma_1
=\left(\begin{array}{cc} 0 & 1\\
1 & 0 \end{array}\right), 
\quad\sigma_2
=\left(\begin{array}{cc} 0 & -i\\
i & 0 \end{array}\right), 
\quad\sigma_3
=\left(\begin{array}{cc} 1 & 0\\
0 & -1 \end{array}\right)
\end{split}
\end{equation}
is the family of {\em Pauli matrices}. Although one can take $m=0$ in the definition of $H$, throughout this article we always assume $m>0$. 

The work \cite{AMV} contains some results concerning $H+V$ for quite general singular measures $\sigma$ in $\R^3$ and suitable $L^2(\sigma)^4$-valued potentials $V$. However, in that article, most of the interest is focused 
on the case that $\sigma$ is the surface measure on the boundary of a bounded regular domain in $\R^3$, and this is our setting for the present paper. In order to make an 	
understandable exposition of the main results, let us first recall the basic ideas used in \cite{AMV}. 

The ambient Hilbert space is
$L^2(\R^3)^4$ with respect to the Lebesgue measure, and $H$ is defined
in the sense of distributions. Given a bounded regular domain $\Omega_+\subset\R^3$ with boundary $\Sigma$ and surface measure $\sigma$, in \cite{AMV} we find domains $D\subset
L^2(\R^3)^4$ in which $H+V:D\to L^2(\R^3)^4$ is an unbounded self-adjoint operator, where $V$ is a suitable $L^2(\sigma)^4$-valued potential. The construction of $D$ relies on the following simple argument:
by assumption, $V$ is $L^2(\sigma)^4$-valued. Thus, given $\varphi\in D$, we can write $V(\varphi)=-g$ in the sense of distributions for some $g\in
L^2(\sigma)^4$. Moreover, since $(H+V)(\varphi)
\in L^2(\R^3)^4$, we can also write $(H+V)(\varphi)=G$ for some $G\in L^2(\R^3)^4$. Therefore, $H(\varphi)=G+g$ in the sense of distributions, and so $\varphi$ should be the
convolution $\phi*(G+g)$, where $\phi$ is a fundamental solution of $H$. In particular,
\begin{equation}\label{intro 1}
\begin{split}
D\subset\{\varphi&=\phi*(G+g): G\in L^2(\R^3)^4,\,g\in L^2(\sigma)^4\}\quad\text{and}\\
&V(\varphi)=-g\quad\text{for all }\varphi=\phi*(G+g)\in D.
\end{split}
\end{equation}
To guarantee that $H+V$ is self-adjoint on $D
$, in \cite{AMV} we impose some relations between $G$ and $g$ with the aid of bounded self-adjoint operators 
$\Lambda:L^2(\sigma)^4\to L^2(\sigma)^4$. That is to say, given suitable $\Lambda$'s, following (\ref{intro 1}) we find domains $D_\Lambda$ where $H+V$ is self-adjoint. 

We consider the potential $V$ given by (\ref{intro 1}) as a ``generic'' potential since it seems to be prescribed from the begining, i.e., $V(\varphi)=-g$ for all $\varphi=\phi*(G+g)\in D_\Lambda$, so $V$ is independent of $\Lambda$.  
For an a priori given potential $V_\sigma$, the key idea in \cite{AMV} to construct a domain where $H+V_\sigma$ is self-adjoint is to find a precise bounded self-adjoint operator $\Lambda_\sigma$ so that $V_\sigma(\varphi)=V(\varphi)$ for all $\varphi\in D_{\Lambda_\sigma}$, and the self-adjointness of $H+V_\sigma$ on $D_{\Lambda_\sigma}$ follows directly from the one of $H+V$. 

The generic potential $V$ given by (\ref{intro 1}) reflects the following rough idea:
if we know that the gradient of a function $\varphi$ has an absolutely continuous part $G$ and a singular part $g$ supported on $\Sigma$ (in our setting, $V(\varphi)\in L^2(\sigma)^4$ and $(H+V)(\varphi)
\in L^2(\R^3)^4$), then $\varphi$ must have a jump across $\Sigma$, and this jump completely determines the singular part of the gradient (in our setting, the jump determines the value $V(\varphi)$). For a given potential $V_\sigma$, one manages to define a suitable domain $D$ such that, for any $\varphi\in D$, the singular part which comes from the gradient on the jump of $\varphi$ across $\Sigma$ agrees with $-V_\sigma(\varphi)$.

As outlined above, in \cite{AMV} we show that any $\varphi=\phi*(G+g)$ has non-tangential boundary values $\varphi_\pm\in L^2(\sigma)^4$ when we approach to $\Sigma$ from $\Omega_\pm$, where $\Omega_-=\R^3\setminus\overline{\Omega_+}$. This enables us to consider, for example, the electrostatic shell potential 
$$V_\lambda(\varphi)=\frac{\lambda}{2}(\varphi_++\varphi_-)$$
for $\lambda\in\R$. Following the argument above, we construct a domain $D_\lambda$ where $H+V_\lambda$ is self-adjoint for all $\lambda\neq\pm2$. Other similar potentials are also treated in \cite{AMV}.

At this point a remark is in order. In \cite{Posi1} (see also \cite[Section 2]{Posi2}) the author provides, in a very general framework, all self-adjoint extensions of symmetric operators obtained by restricting a self-adjoint operator $A$ with domain $D(A)$ to a dense (and closed with respect to the graph norm) subspace $N\subset D(A)$. Typically, $A$ is a differential operator and $N$ is the kernel of some trace operator along a null set. In particular, \cite{Posi1} can be used to provide all the domains where $H+V$ is self-adjoint and, in this sense, some of the results in \cite{AMV} follow from the ones in \cite{Posi1}. However, as regards applications, we make use of some explicit layer potentials on $\Sigma$ derived from our approach in \cite{AMV}  which turn out to be very fruitful both in the description of the domains and in the study of the properties of the couplings under consideration.

Concerning the results in this article, in Section \ref{s spec} we show a general criterion for the existence of eigenvalues in $(-m,m)$ for $H+V$, namely Proposition \ref{spec p1}, which has strong connections with \cite[Theorem 2.5]{Posi2}. This criterion relates the existence of eigenvalues, which is a problem in the whole $\R^3$, with a spectral property of certain bounded operators in $L^2(\sigma)^4$, which is a problem settled exclusively on $\Sigma$. 
Afterwards, we show some applications to the case of electrostatic shell potentials $V_\lambda$. In particular, Theorem \ref{spec t1} shows that $H+V_\lambda$ and $H+V_{-4/\lambda}$ have the same eigenvalues in $(-m,m)$ for all $\lambda\neq0$, which is a special case of the more general isospectral transformation, and that $H+V_\lambda$ has no eigenvalues in $(-m,m)$ if $|\lambda|$ is too big or too small. This is an interesting feature, since it shows that there are lower and upper thresholds on 
the possible values of $|\lambda|$ in order to have non-trivial eigenvalues in $(-m,m)$ for $H+V_\lambda$, unlike what happens to the coupling of $H$ with similar potentials (see Remark \ref{threshold}).
Theorem \ref{symm a -a} is another consequence of the general criterion, where we show that, under a symmetry assumption on $\sigma$, $H+V_\lambda$ has an eigenvalue $a\in(-m,m)$ if and only if $H+V_{-\lambda}$ has $-a$ as an eigenvalue. For completeness, we also show in Theorem \ref{spec t2} that if $\Omega_-$ is connected then $H+V_\lambda$ has no eigenvalues in $\R\setminus[-m,m]$.

Section \ref{s sphere} is devoted to the spectral study of $H+V_\lambda$ when $\Sigma$ is the sphere $$S^2=\{x\in\R^3: |x|=1\}.$$ Our interest is to characterize the eigenvalues finding sharp constants of some appropriately chosen inequalities. In principle, this is far from obvious for Dirac hamiltonians because they are not semibounded operators. However, this procedure has been successfully used in  \cite{DDEV} with Hamiltonians with coulombic singularities. In that case, the sharp constants appear as strictly positive lower bounds for the absolute value of the commutator of two operators, being $H$ one of them. Therefore, this approach can be seen as another use of the uncertainty principle. In the current article we proceed in the same vein and we give an uncertainty principle that concerns some $L^2(\sigma)^2$ bounded operators also related to $H$. In Theorem \ref{thmineq} we obtain a sharp inequality on $S^2$ which turns out to be connected to the eigenvalues of $H+V_\lambda$. For its proof, we strongly use the classical spherical harmonics, so a generalization to other surfaces $\Sigma$ seems an interesting and challenging question. As a consequence of the above-mentioned inequality, we recover a well-known sharp lower bound for the $2$-dimensional Riesz transform on the sphere (see Corollary \ref{riesz} and \cite{HMMPT}). In Lemma \ref{eigenfunction} we provide a specific criterion (based on Proposition \ref{spec p1}) to generate eigenvectors of $H+V_\lambda$. Section \ref{further com} contains some comments on the relation between the uncertainty principle and the eigenvectors of  $H+V_\lambda$, positive results on the existence of eigenvalues, and an open question (as far as we know) and some consequences of an affirmative answer.

Finally, in Section \ref{s confinement} we deal with the confinement phenomenon on regular surfaces. Roughly speaking, one says that an $L^2(\sigma)^4$-valued potential $V$ generates confinement with respect to $H$ and $\Sigma$ if the particles modelized by the evolution operator associated to $H+V$ never cross $\Sigma$ over time. That is, if a function $u(x,t)$ verifies the equation $\partial_t u+i(H+V)u=0$ and $u(\cdot,0)$ is supported in $\Omega_\pm$, that $V$ generates confinement means that $u(\cdot,t)$ is supported in $\Omega_\pm$ for all $t\in\R$, so $\Sigma$ becomes impenetrable for the particles. Similarly to Section \ref{s spec}, in Section \ref{s confinement} we first show a general criterion on $H+V$ to generate confinement, namely Theorem \ref{conf t1}. This criterion is stated in terms of an algebraic property of certain bounded operators in $L^2(\sigma)^4$, so as before we convert a problem in the whole $\R^3$ to a problem exclusively settled in $\Sigma$. An application to electrostatic and Lorentz scalar shell potentials is also shown. In particular, in Theorem \ref{conf appl} we prove that, for $\lambda_e,\lambda_s\in\R$, the potential
$$V_{es}(\varphi)=\frac{1}{2}\,(\lambda_e+\lambda_s\beta)(\varphi_++\varphi_-)$$
generates confinement if and only if $\lambda_e^2-\lambda_s^2=-4$. This was already known for the case of the sphere (see \cite{Dittrich}), and we generalize it to sufficiently regular surfaces.
For the reader only interested on confinement, we mention that Section \ref{s confinement} can be read independently of Sections \ref{s spec} and \ref{s sphere}. Finally, we also want to recall \cite{Posi3}, where the authors studied the confinement phenomenon for singular perturbations of general self-adjoint Hamiltonian operators.

It is worth mentioning that, although all the applications in this article are concerned to the potentials $V_\lambda$ and $V_{es}$ above-mentioned, the general criteria stated in Proposition \ref{spec p1} and Theorem \ref{conf t1} can be used as a first step to study the spectrum and confinement for the coupling of $H$ with other shell potentials. In a sense, once a potential $V_\sigma$ is given, one must find the suitable operator $\Lambda$ (mentioned in the beginning of the introduction) so that $V_\sigma(\varphi)=V(\varphi)$ for all $\varphi\in D_\Lambda$, and then one must check the criteria for that specific $\Lambda$.

\section{Preliminaries}\label{s preliminaries}
Since this article is a continuation of the study developed in \cite{AMV}, we assume that the reader is familiar with the notation, methods and results in there. However, in this section we recall some basic rudiments for the sake of completeness.

Given a positive Borel measure $\nu$ in $\R^3$, set
$$L^2(\nu)^4=\left\{f:\R^3\to\C^4\text{ $\nu$-measurable}:\, 
\int|f|^2\,d\nu<\infty\right\},$$
and denote by $\langle\cdot,\cdot\rangle_{\nu}$ and $\|\cdot\|_\nu$ the standard scalar product and norm in $L^2(\nu)^4$, i.e., 
$\langle f,g\rangle_\nu=\int f\cdot\overline g\,d\nu$ and 
$\|f\|^2_\nu=\int|f|^2\,d\nu$ for $f,g\in L^2(\nu)^4$.  We write $I_4$ or $1$ interchangeably to denote the identity operator on $L^2(\nu)^4$.  We say that $\nu$ is $d$-dimensional if there exists $C>0$ such that $\nu(B(x,r))\leq Cr^d$ for all $x\in\R^3$, $r>0$. We denote by $\mu$ the Lebesgue measure in $\R^3$.

Let $\Sigma$ be the boundary of a bounded Lipschitz domain $\Omega_+\subset\R^3$, let $\sigma$ and $N$ be the surface measure and outward unit normal vector field on $\Sigma$ respectively, and set $\Omega_-=\R^3\setminus\overline{\Omega_+}$, so $\Sigma=\partial\Omega_\pm$. Note that $\sigma$ is 2-dimensional. Since we are not interested in optimal regularity assumptions, for the sequel we assume that $\Sigma$ is of class $\CC^2$ (see Remark \ref{r regularity} for the Lipschitz case).
Finally, we introduce the auxiliary space of measures $$\XX=\left\{G\mu+g\sigma:\,G\in L^2(\mu)^4,\, g\in L^2(\sigma)^4\right\}.$$

Observe that $H$, which is symmetric and initially defined in $\CC^\infty_c(\R^3)^4$ ($\C^4$-valued functions in $\R^3$ which are $\CC^\infty$ and with compact support), can be extended by duality to the space of distributions with respect to the test space $\CC^\infty_c(\R^3)^4$ and, in particular, it can be defined on $\XX$.
The following lemma is concerned with the resolvent of $H$, which will be very useful for the results below. As usual, we denote by $\overline{(\phi^a)^t}$ the complex conjugate of the transpose of $\phi^a$, that is, $((\phi^a)^t)_{\,j,k}=\phi^a_{k,j}$ and $(\overline{\phi^a})_{j,k}=\overline{\phi^a_{j,k}}$ for all $1\leq j,k\leq 4$.
\begin{lema}\label{resolvent}
Given $a\in(-m,m)$, a fundamental solution of $H-a$ is given by 
$$\phi^a(x)=\frac{e^{-\sqrt{m^2-a^2}|x|}}{4\pi|x|}\left(a+m\beta
+\left(1+\sqrt{m^2-a^2}|x|\right)\,i\alpha\cdot\frac{x}{|x|^2}\right)\quad\text{for }x\in\R^3\setminus\{0\},$$
i.e., $(H-a)\phi^a=\delta_0 I_4$ in the sense of distributions, where $\delta_0$ denotes the Dirac measure on $0$. Furthermore, $\phi^a$ satisfies $(i)$, $(ii)$, and $(iii)$ of \cite[Section 2.2]{AMV}, that is,
\begin{itemize}
\item[$(i)$] $\phi^a_{j,k}\in\CC^\infty(\R^3\setminus\{0\})$ for all $1\leq j,k\leq 4$,
\item[$(ii)$] $\phi^a(x-y)=\overline{(\phi^a)^t}(y-x)$ for all $x,y\in\R^3$ such that $x\neq y$, 
\item[$(iii)$] there exist $\gamma,\delta>0$ such that
\begin{itemize}
\item[$(a)$] $\sup_{1\leq j,k\leq 4}|\phi^a_{j,k}(x)|\leq C|x|^{-2}$ for all $|x|<\delta$,
\item[$(b)$] $\sup_{1\leq j,k\leq 4}|\phi^a_{j,k}(x)|\leq Ce^{-\gamma|x|}$ for all $|x|>1/\delta$,
\item[$(c)$] $\sup_{1\leq j,k\leq 4}\,
\sup_{\xi\in\R^3}(1+|\xi|^2)^{1/2}|\FF(\phi^a_{j,k})(\xi)|<\infty,$ where $\FF$ denotes the Fourier transform in $\R^3$.
\end{itemize}
\end{itemize}
\end{lema}

\begin{proof}
It is well known that $\psi(x)=e^{-\sqrt{m^2-a^2}|x|}(4\pi|x|)^{-1}$ is a fundamental solution of the Helmhotz operator $-\Delta+m^2-a^2$ in $\R^3$. Note that $H^2=(-\Delta+m^2)I_4$, so
$$(H-a)(H+a)=H^2-a^2=(-\Delta+m^2-a^2)I_4,$$ and if we set $\phi^a(x)=(H+a)(\psi(x)I_4)$ then $(H-a)\phi^a=\delta_0I_4$ in the sense of distributions. The explicit formula for $\phi^a$ follows by a straightforward computation. The rest of the proof is analogous to \cite[Lemma 3.1]{AMV}.
\end{proof}

Given a positive Borel measure $\nu$ in $\R^3$, $f\in L^2(\nu)^4$, and $x\in\R^3$, we set $$(\phi^a*f\nu)(x)=\int\phi^a(x-y)f(y)\,d\nu(y),$$ whenever the integral makes sense. Actually, by Lemma \ref{resolvent} and \cite[Lemma 2.1]{AMV}, 
if $\nu$ is a $d$-dimensional measure in $\R^3$ with $1<d\leq 3$, then there exists some constant $C>0$ such that $\|\phi^a*g\nu\|_{\mu}\leq C\|g\|_{\nu}$ for all $g\in L^2(\nu)^4$.

In what follows we use a non standard notation, $\Phi^a$, to define the convolution of measures in $\XX$ with the fundamental solution of $H-a$, $\phi^a$. Capital letters, as $F$ or $G$, in the argument of $\Phi^a$ denote elements of $L^2(\mu)^4$, and the lowercase letters, as $f$ or $g$, denote elements in  $L^2(\sigma)^4$. Despite that this notation is non standard, it is very convenient in order to shorten the forthcoming computations. 

Given $G\mu+g\sigma\in\XX$, we define $$\Phi^a(G+g)=
\phi^a*G\mu+\phi^a*g\sigma.$$
Then, the above inequality shows that 
$\|\Phi^a(G+g)\|_\mu\leq C(\|G\|_\mu+\|g\|_\sigma)$ for some constant $C>0$ and all $G\mu+g\sigma\in\XX$, so $\Phi^a(G+g)\in L^2(\mu)^4$. Moreover, following \cite[Section 2.3]{AMV} one can show that 
$(H-a)(\Phi^a(G+g))=G\mu+g\sigma$ in the sense of distributions. This allows us to define a ``generic'' potential $V$ acting on functions $\varphi=\Phi^a(G+g)$ by
\begin{equation*}
V(\varphi)= -g\sigma,
\end{equation*}
so that $(H-a+V)(\varphi)=G\mu$
in the sense of distributions. For simplicity of notation, we write 
$(H-a+V)(\varphi)=G\in L^2(\mu)^4$. 

In order to construct a domain of definition where $H+V$ is self-adjoint, in \cite{AMV} we had to consider the trace operator on $\Sigma$.  
For $G\in\CC_c^\infty(\R^3)^4$, one defines the trace operator on $\Sigma$ by $\Tr_\Sigma(G)=G\chi_{\Sigma}$. Then, $\Tr_\Sigma$ extends to a bounded linear operator $$\Tr_\sigma:W^{1,2}(\mu)^4\to L^2(\sigma)^4$$ (see \cite[Proposition 2.6]{AMV}, for example), where $W^{1,2}(\mu)^4$ denotes the Sobolev space of $\C^4$-valued functions such that all its components have all their derivatives up to first order in $L^2(\mu)$.  From Lemma \ref{resolvent}, we have
$\|\Phi^a(G)\|_{W^{1,2}(\mu)^4}\leq C\|G\|_{\mu}$ for some $C>0$ and all $G\in L^2(\mu)^4$ (see \cite[Lemma 2.8]{AMV}), thus we can define $$\Phi_\sigma^a(G)=\Tr_\sigma(\Phi^a(G))=\Tr_\sigma(\phi^a*G\mu)$$ and it satisfies
$\|\Phi_\sigma^a(G)\|_\sigma\leq C\|G\|_\mu$ for all $G\in L^2(\mu)^4$. 
Note that, for $a=0$, the above definitions recover the ones in \cite[Section 2.3]{AMV}.

The next lemma, which is a generalization of \cite[Lemma 3.3]{AMV}, will be used in the sequel.
\begin{lema}\label{l jump}
Given $g\in L^2(\sigma)^4$, $x\in\Sigma$ and $a\in\R$, set 
\begin{equation*}
\begin{split}
C_\sigma^a (g)(x)=\lim_{\epsilon\searrow0}\int_{|x-z|>\epsilon}\phi^a(x-z)g(z)\,d\sigma(z) 
\quad\text{and}\quad
C_{\pm}^a(g)(x)=\lim_{\Omega_{\pm}\ni y\stackrel{nt}{\longrightarrow} x}
\Phi^a(g)(y),
\end{split}
\end{equation*}
where $\textstyle{\Omega_{\pm}\ni y\stackrel{nt}{\longrightarrow} x}$ means that $y\in\Omega_{\pm}$ tends to $x\in\Sigma$ non-tangentially. Then $C_\sigma^a$ and $C_\pm^a$ are bounded linear operators in $L^2(\sigma)^4$. Moreover, the following  holds:
\begin{itemize}
\item[$(i)$] $C_\pm^a =\mp\frac{i}{2}\,(\alpha\cdot N)+C_\sigma^a$ (Plemelj--Sokhotski jump formulae),
\item[$(ii)$] $-4(C_\sigma^a(\alpha\cdot N))^2=-4((\alpha\cdot N)C_\sigma^a)^2=I_4$ for $a\in[-m,m]$.
\end{itemize}
\end{lema}

\begin{proof}
The proof of the lemma is completely analogous to the one of \cite[Lemma 3.3]{AMV}, so we omit it. Concerning the second term in $(ii)$, once we know that $-4(C_\sigma^a(\alpha\cdot N))^2=I_4$, then, multiplying the equation by $\alpha\cdot N$ from the left and from the right and using that $(\alpha\cdot N)^2=I_4$, we obtain $-4((\alpha\cdot N)C_\sigma^a)^2=I_4$.
\end{proof}

In accordance with the notation introduced in \cite{AMV}, for the case $a=0$, we write $\Phi$, $\Phi_\sigma$, $C_\pm$ and $C_\sigma$ instead of $\Phi^0$, $\Phi_\sigma^0$, $C_\pm^0$ and $C_\sigma^0$, respectively.

Finally, we recall our main tool to construct domanis where $H+V$ is self-adjoint, namely \cite[Theorem 2.11]{AMV}. Actually, the following theorem is a direct application of \cite[Theorem 2.11]{AMV} to $H+V$, and we state it here in order to make the exposition more self-contained. Given an operator between vector spaces $S:X\to Y$, denote 
$$\Ker(S)=\{x\in X:\, S(x)=0\}\quad\text{and}\quad
\Ran(S)=\{S(x)\in Y:\, x\in X\}.$$
 
\begin{teo}\label{pre t1}
Let $\Lambda:L^2(\sigma)^4\to L^2(\sigma)^4$ be a bounded operator. Set $$D(T)=\{\Phi(G+g): G\mu+g\sigma\in\XX,\,\Phi_\sigma(G)=\Lambda(g)\}\subset L^2(\mu)^4
\text{ and $T=H+V$ on $D(T)$,}$$ where $V(\varphi)=-g\sigma$ and $(H+V)(\varphi)=G$ for all $\varphi=\Phi(G+g)\in D(T)$. If $\Lambda$ is self-adjoint and $\Ran(\Lambda)$ is closed, then $T:D(T)\to L^2(\mu)^4$ is an essentially self-adjoint operator. In that case, if $\{{\Phi(h)}:\,h\in\Ker(\Lambda)\}$ is closed in $L^2(\mu)^4$, then $T$ is self-adjoint.
\end{teo}

In particular, if $\Lambda$ is self-adjoint and semi-Fredholm (see for example \cite[Definition 1.40]{Aiena}), then the operator $T$ in Theorem \ref{pre t1} is self-adjoint. Recall also that any bounded, semi-Fredholm and self-ajoint operator on a Hilbert space is actually Fredholm.

\begin{remark}\label{r regularity}
All the results in this article which are proved without the use of Fredholm's theorem are valid when $\Sigma$ is just Lipschitz (but not necessarily of class $\CC^2$), or when it is the graph of a Lipschitz function from $\R^2$ to $\R$. Actually, the smoothness and boundedness assumptions on $\Sigma$ are exclusively required for compactness purposes, in order to use Fredholm's theory.
\end{remark}

\section{On the point spectrum}\label{s spec}
In this section, we show a criterion for the existence of eigenvalues in $(-m,m)$ for $H+V$, namely Proposition \ref{spec p1}. This criterion relates the eigenvalues with a spectral property of certain bounded operators in $L^2(\sigma)^4$. Afterwards, we show some applications to the case of electrostatic shell potentials.

\begin{propo}\label{spec p1}
Let $T$ be as in Theorem \ref{pre t1}. Given $a\in(-m,m)$, there exists $\varphi=\Phi(G+g)\in D(T)$ such that $T(\varphi)=a\varphi$ if and only if $\Lambda(g)=(C_\sigma^a-C_\sigma)(g)$ and $G=a\Phi^a(g)$.
Therefore, $\Ker(T-a)\neq0$ if and only if $\Ker(\Lambda+C_\sigma-C_\sigma^a)\neq0$.
\end{propo}

\begin{proof}
Let $a\in(-m,m)$ and assume that $T(\varphi)=a\varphi$ for some $\varphi=\Phi(G+g)\in D(T)$. Then, 
\begin{equation}\label{spec p1 eq1}
G=(H+V)(\varphi)=a\varphi=a\Phi(G+g),
\end{equation}
and therefore $H(G)=aG\mu+ag\sigma$ in the sense of distributions. This yields $(H-a)(G)=ag\sigma$, and applying $\Phi^a$ we conclude that $G=a\Phi^a(g)$. In particular, we have seen that $a\Phi(G+g)=a\Phi^a(g)$, and thus Lemma \ref{l jump}$(i)$ yields
\begin{equation*}
a\left(\Lambda\mp\frac{i}{2}\,(\alpha\cdot N)+C_\sigma\right)(g)=a(\Phi_\sigma(G)+C_\pm(g))=aC_\pm^a(g)
=a\left(\mp\frac{i}{2}\,(\alpha\cdot N)+C_\sigma^a\right)(g).
\end{equation*}
Summing these equations, we obtain $a(\Lambda+C_\sigma)(g)=aC_\sigma^a(g)$, which is equivalent to $\Lambda(g)=(C_\sigma^a-C_\sigma)(g)$ when $a\neq0$. For the case $a=0$, from (\ref{spec p1 eq1}) we have that $G=0$ and, since $\varphi\in D(T)$, 
$\Lambda(g)=\Phi_\sigma(G)=0=(C_\sigma-C_\sigma^0)(g)$.

On the contrary, assume that there exists $g\in L^2(\sigma)^4$ such that $\Lambda(g)=(C_\sigma^a-C_\sigma)(g)$. Set $G=a\Phi^a(g)$ and $\varphi=\Phi(G+g)$. As we did in (\ref{spec p1 eq1}), $G=a\Phi^a(g)$ implies $(H-a)(G)=ag\sigma$, so $G=a\Phi(G+g)$ and
$$T(\varphi)=(H+V)(\varphi)=G=a\Phi(G+g)=a\varphi.$$
In particular, $a\Phi^a(g)=a\Phi(G+g)$, thus Lemma \ref{l jump}$(i)$ once again gives 
\begin{equation*}
a\left(\mp\frac{i}{2}\,(\alpha\cdot N)+C_\sigma^a\right)(g)
=a\Phi_\sigma(G)+a\left(\mp\frac{i}{2}\,(\alpha\cdot N)+C_\sigma\right)(g).
\end{equation*}
For $a\neq0$, this and the assumption on $g$ imply $\Phi_\sigma(G)=(C_\sigma^a-C_\sigma)(g)=\Lambda(g)$, thus $\varphi\in D(T)$.
As before, the case $a=0$ can be easily treated separately. The proposition is proved.
\end{proof}

\subsection{Electrostatic shell potentials}\label{ss electro}
In \cite[Theorem 3.8]{AMV} we proved that, 
if $\lambda\in\R\setminus\{0,\pm2\}$ and $T$ is the operator defined by
\begin{equation}\label{domain}
\begin{split}
D(T)=\big\{\Phi&(G+g): G\mu+g\sigma\in\XX,\,\Phi_\sigma(G)=\Lambda(g)\big\}\\
&\text{and}\quad T=H+V_\lambda \text{ on } D(T),
\end{split}
\end{equation} 
where 
$$\Lambda=-(1/\lambda+C_\sigma),\qquad V_\lambda(\varphi)=\frac{\lambda}{2}(\varphi_++\varphi_-)\sigma$$ and 
$\varphi_\pm=\Phi_\sigma(G)+C_\pm (g)$ for $\varphi=\Phi(G+g)\in D(T)$, then $T:D(T)\subset L^2(\mu)^4\to L^2(\mu)^4$ is self-adjoint. Moreover, we also showed that $V_\lambda=V$ on $D(T)$ for all $\lambda\neq0$, so the self-adjointness was a consequence of Theorem \ref{pre t1}.

\begin{lema}\label{spec l1}
Set $\|C_\sigma^a\|=\inf\left\{C>0:\|C_\sigma^a (g)\|_\sigma\leq C\|g\|_\sigma\text{ for all }g\in L^2(\sigma)^4\right\}.$ Then, we have $\sup_{a\in(-m,m)}\|C_\sigma^a\|<\infty.$
\end{lema}

\begin{proof}
We write
\begin{equation*}
\begin{split}
\phi^a(x)&=\frac{e^{-\sqrt{m^2-a^2}|x|}}{4\pi|x|}\left(a+m\beta
+i\sqrt{m^2-a^2}\,\alpha\cdot\frac{x}{|x|}\right)+\frac{e^{-\sqrt{m^2-a^2}|x|}-1}{4\pi}\,i\left(\alpha\cdot\frac{x}{|x|^3}\right)\\
&\quad+\frac{i}{4\pi}\,\left(\alpha\cdot\frac{x}{|x|^3}\right)
=\omega_1(x)+\omega_2(x)+\omega_3(x).
\end{split}
\end{equation*}
Note that 
\begin{equation}\label{spec l1 eq1}
\sup_{a\in(-m,m)}\sup_{1\leq k,l\leq 4}|(\omega_1)_{k,l}(x)|=O(|x|^{-1})\quad\text{for }|x|\to0,
\end{equation}
and by the mean value theorem we have the same estimate for $\omega_2$. Using (\ref{spec l1 eq1}), that $\sigma$ is 2-dimensional and rather standard arguments (essentially, that $\Sigma$ is bounded and the generalized Young's inequality), it is easy to show that the convolution operator with kernel $\omega_1+\omega_2$ is bounded in $L^2(\sigma)^4$ uniformly on $a\in(-m,m)$. Finally, the $L^2(\sigma)^4$ boundedness of the singular intergal operator with convolution kernel $\omega_3$ follows, for example, by \cite[Theorem 20.15]{Mattila-llibre}, working component by component. Note that this last operator does not depend on $a$, so the lemma is proved.
\end{proof}

\begin{teo}\label{spec t1}
Let $\lambda\in\R\setminus\{0\}$, let $T=H+V_\lambda$ be as in $(\ref{domain})$, and $a\in(-m,m)$. Then $\Ker(T-a)\neq0$ if and only if $\Ker(1/\lambda+C_\sigma^a)\neq0$. In particular, $H+V_\lambda$ and $H+V_{-4/\lambda}$ have the same eigenvalues in $(-m,m)$.

Moreover, if $|\lambda|\not\in[1/\|C_\sigma^a\|,4\|C_\sigma^a\|]$ then $a$ is not an eigenvalue of $H+V_\lambda$, and if
$|\lambda|\not\in[1/C,4C]$ then $H+V_\lambda$ has no eigenvalues in $(-m,m)$, where
$C=\sup_{a\in(-m,m)}\|C_\sigma^a\|<\infty$.
\end{teo}

\begin{proof}
That $\Ker(T-a)\neq0$ if and only if $\Ker(1/\lambda+C_\sigma^a)\neq0$ is a direct consequence of the definition of $\Lambda$ and Proposition \ref{spec p1}. 

Assume that $\Ker(1/\lambda+C_\sigma^a)\neq0$. Then there exists a non-trivial $g\in L^2(\sigma)^4$ such that $C_\sigma^a(g)=-g/\lambda$. Using Lemma \ref{l jump}$(ii)$ we deduce that 
$-\frac{1}{4}g=-\frac{1}{\lambda}((\alpha\cdot N)C_\sigma^a(\alpha\cdot N))(g)$. This easily implies that $C_\sigma^a(f)=\frac{\lambda}{4}f$ for $f=(\alpha\cdot N)g\neq0$, so $\Ker(-\lambda/4+C_\sigma^a)\neq0$. The same arguments actually show that $\Ker(1/\lambda+C_\sigma^a)\neq0$ if and only if $\Ker(-\lambda/4+C_\sigma^a)\neq0$, and by the first part of the theorem, $H+V_\lambda$ and $H+V_{-4/\lambda}$ have the same eigenvalues in $(-m,m)$.

For the last part of the theorem, since $(\alpha\cdot N)^2=I_4$, we easily have $\|(\alpha\cdot N)g\|_\sigma\leq\|g\|_\sigma$ for all $g\in L^2(\sigma)^4$. Combining this estimate with Lemma \ref{l jump}$(ii)$, we obtain
\begin{equation*}
\frac{1}{4}\|g\|_\sigma=\|((\alpha\cdot N)C_\sigma^a)^2(g)\|_\sigma
\leq\|C_\sigma^a\|\|C_\sigma^a(g)\|_\sigma.
\end{equation*}
Hence,
\begin{equation}\label{spec t1 eq1}
\frac{1}{4C}\leq\frac{1}{4\|C_\sigma^a\|}
\leq\frac{\|C_\sigma^a(g)\|_\sigma}{\|g\|_\sigma}
\leq\|C_\sigma^a\|\leq C
\end{equation}
if $g\neq0$. By the first part of the theorem, if there exists $\varphi\in D(T)$ such that $(H+V_\lambda)(\varphi)=a\varphi$ for some $a\in(-m,m)$, then there exists a non-trivial $g\in L^2(\sigma)^4$ such that $C_\sigma^a(g)=-g/\lambda$. Thus, (\ref{spec t1 eq1}) easily implies
$|\lambda|\in[1/\|C_\sigma^a\|,4\|C_\sigma^a\|]\subset[1/C,4C]$, and the theorem follows.
\end{proof}

\begin{remark}\label{threshold}
Theorem \ref{spec t1} shows that the coupling of the free Dirac operator $H$ with electrostatic shell potentials $V_\lambda$ relative to $\Sigma$ does not generate eigenvalues either for big or small values of $|\lambda|$. Recall that the coupling of $H$ with the Coulomb potential $\lambda/|x|$ generates eigenvalues for any small $|\lambda|$ (see \cite[Theorem 1]{DDEV}, for example) and is not essentially self-adjoint if $|\lambda|$ is big enough. On the other hand, it is not hard to see that there exists a sequence $\{\lambda_j\}_{j\in\N}\subset\R$ with $|\lambda_j|\to\infty$ for $j\to\infty$ such that the coupling of $H$ with the potentials $\lambda_j\chi_{B(0,1)}$ generates eigenvalues.
\end{remark}

\begin{remark}\label{lambda-a}
If we define $\Lambda_\pm^a=1/\lambda\pm C_\sigma^a$, 
by Lemma \ref{l jump}$(ii)$ we have
\begin{equation}\label{spec t1 eq2}
\Lambda_+^a\Lambda_-^a=\frac{1}{\lambda^2}-(C_\sigma^a)^2
=\frac{1}{\lambda^2}-\frac{1}{4}
-C_\sigma^a(\alpha\cdot N)\{\alpha\cdot N,C_\sigma^a\}=b-K,
\end{equation}
where $b=1/\lambda^2-1/4$ and $K=C_\sigma^a(\alpha\cdot N)\{\alpha\cdot N,C_\sigma^a\}$.
Following the arguments of \cite[Lemma 3.5]{AMV}, one can show that $\{\alpha\cdot N,C_\sigma^a\}$ is a compact operator, as well as $K$. Moreover, $K$ is easily seen to be self-adjoint, and hence it has a non-trivial eigenfunction. Therefore, for any $a\in(-m,m)$ there exists some $\lambda$ such that $\Ker(1/\lambda+C_\sigma^a)\neq0$ by (\ref{spec t1 eq2}), so the second part of Theorem \ref{spec t1} is meaningful.
\end{remark}

Note that (\ref{spec t1 eq1}) yields $\|C_\sigma^a\|\geq1/2$ for all $a\in(-m,m)$. In particular, this lower bound of $\|C_\sigma^a\|$ does not depend on $\Sigma$. For an upper bound, this type of result may not be expected because, roughly speaking, the abruptness of $\Sigma$ may play a role in questions concerning upper bounds for the norm of singular integral operators on Lipschitz surfaces (see for example \cite[Chapter 20]{Mattila-llibre} for related topics).

The following theorem generalizes some results of \cite[Theorem 3.8$(ii)$]{AMV}.
\begin{teo}\label{symm a -a}
Assume that $\sigma=s_\#\sigma$, where $s(x)=-x$ for $x\in\R^3$ and $s_\#\sigma$ is the image measure of $\sigma$ with respect to $s$.
Let $\lambda\in\R\setminus\{0\}$ and $H+V_\lambda$ be as in $(\ref{domain})$. If $H+V_\lambda$ has some eigenvalue $a\in(-m,m)$, then $H+V_{-\lambda}$ has $-a$ as an eigenvalue. 
\end{teo}
\begin{proof}
From Theorem \ref{spec t1} we see that, if $H+V_\lambda$ has $a$ as an eigenvalue, there exists a nontrivial $g\in L^2(\sigma)^4$ such that $(1/\lambda+C_\sigma^a)(g)=0$. Set $f=-\tau (g\circ s)$, where 
$$\tau
=\left(\begin{array}{cc} 0 & I_2\\
I_2 & 0 \end{array}\right)$$
and $I_2$ denotes the identity operator in $L^2(\sigma)^2$. Obviously, $f\neq0$ and, since $\sigma=s_\#\sigma$, we have $f\in L^2(\sigma)^4$.
It is straightforward to check that 
$-\phi^{-a}(z)\tau=\tau\phi^{a}(-z)$ for all $z\in\R^3\setminus\{0\}$. Therefore, using the assumptions on $\sigma$ and on $g$,
\begin{equation*}
\begin{split}
C^{-a}_\sigma(f)(x)&=\lim_{\epsilon\searrow0}\int_{|x-y|>\epsilon}-\phi^{-a}(x-y)\tau g(-y)\,d\sigma(y)\\
&=\tau\lim_{\epsilon\searrow0}\int_{|x-y|>\epsilon}\phi^a(-x+y)g(-y)\,ds_\#\sigma(y)\\
&=\tau\lim_{\epsilon\searrow0}\int_{|x+y|>\epsilon}\phi^a(-x-y) g(y)\,d\sigma(y)\\
&=\tau C^a_\sigma(g)(-x)
=-\frac{1}{\lambda}\,\tau g(-x)=\frac{1}{\lambda}\,f(x).
\end{split}
\end{equation*}
That is,  $(-1/\lambda+C_\sigma^{-a})(f)=0$ for some $f\neq0$. By Theorem \ref{spec t1} once again, $H+V_{-\lambda}$ has $-a$ as an eigenvalue.
\end{proof}

\begin{teo}\label{spec t2}
Let $\lambda\in\R\setminus\{0\}$ and let $T=H+V_\lambda$ be as in $(\ref{domain})$. If $\Omega_-$ is connected, then $T$ has no eigenvalues in $\R\setminus[-m,m]$.
\end{teo}

\begin{proof}
Let $a\in\R\setminus[-m,m]$ and $\varphi\in D(T)$ such that $T(\varphi)=a\varphi$. We will see that $\varphi=0$. Note that $(H-a)(\varphi)=0$ in $\Omega_-$ and hence $(-\Delta+m^2-a^2)(\varphi)=(H+a)(H-a)(\varphi)=0$. Therefore, $(\Delta+k^2)(\varphi)=0$ in $\Omega_-$ for $k=\sqrt{a^2-m^2}>0$.

Since $\varphi\in D(T)$, then $\varphi=\Phi(G+g)$ for some $G\mu+g\sigma\in\XX$ such that $\Phi_\sigma(G)=\Lambda(g)$. Using that $m>0$, it is not hard to show that 
\begin{equation*}
\int_{S^2_R}|\varphi|^2\,d\sigma_R=o(1)\qquad\text{as }R\to\infty,
\end{equation*}
where $S^2_R=\{x\in\R^3:|x|=R\}$ and $\sigma_R$ denotes the surface measure on $S^2_R$. Therefore, Rellich's lemma yields $\varphi=0$ in a neighbourhood of infinity (see \cite{Rellich}, for example), and thus $\varphi=0$ in $\Omega_-$ by unique continuation and the connectivity assumption. In particular 
$\varphi_-=0$ in $\Sigma$, and so Lemma \ref{l jump}$(i)$ and the definition of $\Lambda$ give
\begin{equation}\label{spec t2 eq2}
\begin{split}
\frac{\lambda}{2}\,\varphi_+&=\frac{\lambda}{2}(\varphi_++\varphi_-)=\lambda(\Phi_\sigma(G)+C_\sigma(g))\\
&=\lambda(\Lambda+C_\sigma)(g)
=-g=i(\alpha\cdot N)(\varphi_--\varphi_+)=-i(\alpha\cdot N)\varphi_+.
\end{split}
\end{equation}
This means that $\left(\frac{\lambda}{2}+i(\alpha\cdot N)\right)\varphi_+=0$, so
$$\left(\frac{\lambda^2}{4}+1\right)\varphi_+=\left(\frac{\lambda}{2}-i(\alpha\cdot N)\right)\left(\frac{\lambda}{2}+i(\alpha\cdot N)\right)\varphi_+=0$$
and then $\varphi_+=0$ in $\Sigma$. It only remains to check that $\varphi=0$ in $\Omega_+$, but this goes along well known lines. Since $T(\varphi)=a\varphi$, then $(H-a)(\varphi)=0$ in $\Omega_+$. If one integrates by parts on the identities
$$0=\int_{\Omega_+\setminus B(x,\epsilon)}(H-a)(\varphi)(z)\cdot\overline{\phi^a(z-x)e_j}\,d\mu(z)\quad\text{for }j=1,2,3,4,$$
where $e_1=(1,0,0,0),\ldots, e_4=(0,0,0,1)$ and $B(x,\epsilon)$ is the ball centered at $x\in\Omega_+$ and with radius $\epsilon>0$, and then one takes $\epsilon\searrow0$, one can show that
\begin{equation}\label{spec t2 eq1}
\varphi(x)=\int_\Sigma\widetilde{\phi^a}(x-z)(i\alpha\cdot N(z))\varphi_+(z)\,d\sigma(z)\quad\text{for }x\in\Omega_+,
\end{equation} 
where 
$$\widetilde{\phi^a}(y)=\frac{e^{i\sqrt{a^2-m^2}|y|}}{4\pi|y|}\left(a+m\beta
+\left(1-i\sqrt{a^2-m^2}|y|\right)\,i\alpha\cdot\frac{y}{|y|^2}\right)\quad\text{for }y\in\R^3\setminus\{0\}.$$
Since $\varphi_+=0$ in $\Sigma$, we conclude from (\ref{spec t2 eq1}) that $\varphi$ vanishes in $\Omega_+$, and thus $\varphi=0$.
\end{proof}

Although the definition of $\Lambda$ does not make sense for $\lambda=0$, one can replace $\Phi_\sigma(G)=\Lambda(g)$ by $\lambda\Phi_\sigma(G)=-(1+\lambda C_\sigma)(g)$ in the definition of $D(T)$ in (\ref{domain}), in order to recover the free Dirac operator $H$ when $\lambda=0$. It is well known that $H$ does not have eigenvalues in $\R\setminus[-m,m]$, so Theorem \ref{spec t2} also holds for $\lambda=0$.

With the same arguments, one can check that Theorem \ref{spec t2} also holds for other $L^2(\sigma)^4$-valued potentials different from $V_\lambda$, as far as they yield a suitable $\Lambda$ for which a relation like (\ref{spec t2 eq2}) implies that $\varphi_+=0$ in $\Sigma$.

\section{The sphere: point spectrum and related inequalities}\label{s sphere}
This section is focused on the coupling $H+V_\lambda$ given in Section \ref{ss electro} (see (\ref{domain})), and mostly in the case that $\Sigma$ is the sphere. However, the following two lemmata hold for any $\Sigma$ and $\sigma$ as in Section \ref{s preliminaries}. First, we need some definitions.

For $a\in[-m,m]$ and $\widetilde\sigma=(\sigma_1,\sigma_2,\sigma_3)$, where the $\sigma_j$'s are the family of Pauli matrices (see (\ref{freedirac})), define 
\begin{equation*}
k^a(x)=\frac{e^{-\sqrt{m^2-a^2}|x|}}{4\pi|x|}\,I_2
\quad\text{and}\quad
w^a(x)=\frac{e^{-\sqrt{m^2-a^2}|x|}}{4\pi|x|^3}
\left(1+\sqrt{m^2-a^2}|x|\right)\,i\,\widetilde{\sigma}\cdot x
\end{equation*}
for $x\in\R^3\setminus\{0\}$. Given $f\in L^2(\sigma)^2$ and $x\in \Sigma$, set
\begin{equation*}
K^a(f)(x)=\int k^a(x-z)f(z)\,d\sigma(z) 
\quad\text{and}\quad
W^a(f)(x)=\lim_{\epsilon\searrow0}\int_{|x-z|>\epsilon} w^a(x-z)f(z)\,d\sigma(z).
\end{equation*}
That $K^a$ and $W^a$ are bounded operators in $L^2(\sigma)^2$ can be verified similarly to the case of $C^a_\sigma$ in $L^2(\sigma)^4$, we omit the details. Moreover, note that 
\begin{equation}\label{sphere eq1}
C_\sigma^a
=\left(\begin{array}{cc}  (a+m)K^a& W^a\\
W^a & (a-m)K^a \end{array}\right).
\end{equation}

\begin{lema}\label{sphere l1}
For $a\in[-m,m]$, the following hold:
\begin{itemize}
\item[$(i)$] the anticommutator $\left\{(\widetilde\sigma\cdot N)K^a,(\widetilde\sigma\cdot N)W^a\right\}$ vanishes identically,
\item[$(ii)$] $((\widetilde\sigma\cdot N)W^a)^2+(a^2-m^2)((\widetilde\sigma\cdot N)K^a)^2=-1/4$.
\end{itemize}
\end{lema}
\begin{proof}
From Lemma \ref{l jump}$(ii)$ and 
(\ref{sphere eq1}) we have
\begin{equation*}
\begin{split}
&\left(\begin{array}{cc}  -1/4&0\\
0 & -1/4 \end{array}\right)
=\left(\left(\begin{array}{cc}  0& \widetilde\sigma\cdot N\\
\widetilde\sigma\cdot N & 0 \end{array}\right)
\left(\begin{array}{cc}  (a+m)K^a& W^a\\
W^a & (a-m)K^a \end{array}\right)\right)^2\\
&=\left(\begin{array}{cc}  ((\widetilde\sigma\cdot N)W^a)^2+(a^2-m^2)((\widetilde\sigma\cdot N)K^a)^2 & 
(a-m)\{(\widetilde\sigma\cdot N)K^a,(\widetilde\sigma\cdot N)W^a\} \\
(a+m)\{(\widetilde\sigma\cdot N)K^a,(\widetilde\sigma\cdot N)W^a\} & ((\widetilde\sigma\cdot N)W^a)^2+(a^2-m^2)((\widetilde\sigma\cdot N)K^a)^2 \end{array}\right),
\end{split}
\end{equation*}
and the lemma follows.
\end{proof}

\begin{lema}\label{sphere l2}
$K^a$ is a positive operator in $L^2(\sigma)^2$ for all $a\in(-m,m)$. 
\end{lema}

\begin{proof}
We want to verify that $\int K^a(f)\cdot \overline f\,d\sigma\geq0$ for all $f\in L^2(\sigma)^2$. If we set 
$u(x)=\int k^a(x-z)f(z)\,d\sigma(z)$ for $x\in\R^3$, it is not hard to check that $u\in L^2(\mu)^2$ and that it satisfies
\begin{equation}\label{sphere l2 eq1}
\left\{
\begin{split}
&(-\Delta+m^2-a^2)u=0\quad\text{in }\R^3\setminus\Sigma,\\
&\Tr_\sigma(u)=K^a(f)\in L^2(\sigma)^2.
\end{split}
\right.
\end{equation}
Moreover, since $\nabla k^a(x)=-k^a(x)(1+\sqrt{m^2-a^2}|x|)x/|x|^2$ for all $x\in\R^3\setminus\{0\}$, a proof analogous to the one of \cite[Lemma 3.3$(i)$]{AMV} shows that
\begin{equation}\label{sphere l2 eq2}
(\nabla u)_+\cdot N-(\nabla u)_-\cdot N= f,
\end{equation}
where $(\nabla u)_\pm$ denote the boundary values of $\nabla u$ when one approaches to $\Sigma$ non-tangentially from $\Omega_\pm$. Therefore, using (\ref{sphere l2 eq1}), (\ref{sphere l2 eq2}), and the divergence theorem, we conclude
\begin{equation*}
\begin{split}
\int K^a(f)\cdot \overline f\,d\sigma&=\int_{\Sigma} \Tr_\sigma(u)\cdot\overline{(\left(\nabla u)_+\cdot N-(\nabla u)_-\cdot N\right)}\,d\sigma\\
&=\int_{\Omega_+}(|\nabla u|^2+u\Delta\overline{u})\,d\mu
+\int_{\Omega_-}(|\nabla u|^2+u\Delta\overline{u})\,d\mu\\
&=\int_{\Omega_+}|\nabla u|^2\,d\mu
+\int_{\Omega_-}|\nabla u|^2\,d\mu+(m^2-a^2)\int_{\R^3}|u|^2\,d\mu
\geq0.
\end{split}
\end{equation*}
\end{proof}

\subsection{An uncertainty principle on the sphere}
Throughout this section we set $\Omega_+=\{x\in\R^3:|x|<1\}$, $\Sigma=S^2$, $\sigma$ denotes the surface measure on $S^2$, and $N(x)=x$ for $x\in S^2$. We also use the notation of \cite[Section 4.6.4]{Thaller}. 

Let $Y^l_n$ be the usual spherical harmonics. They are defined for $n=0,1,2,\ldots$, and $l=-n,-n+1,\ldots,n$, and they satisfy $\Delta_{S^2}Y^l_n=n(n+1)Y^l_n$, where $\Delta_{S^2}$ denotes the usual spherical laplacian. Moreover, $Y^l_n$ form a complete orthonormal set in $L^2(\sigma)$.

For $j=1/2,3/2,5/2,\ldots,$ and $m_j=-j,-j+1,\ldots,j$, set
\begin{equation*}
\begin{split}
\psi_{j-1/2}^{m_j}&=\frac{1}{\sqrt{2j}}
\left(\begin{array}{c} \sqrt{j+m_j}\, Y_{j-1/2}^{m_j-1/2}\\
\sqrt{j-m_j}\, Y_{j-1/2}^{m_j+1/2} \end{array}\right)
\quad\text{and}\\
\psi_{j+1/2}^{m_j}&=\frac{1}{\sqrt{2j+2}}
\left(\begin{array}{c} \sqrt{j+1-m_j}\, Y_{j+1/2}^{m_j-1/2}\\
-\sqrt{j+1+m_j}\, Y_{j+1/2}^{m_j+1/2} \end{array}\right).
\end{split}
\end{equation*}
Then $\psi_{j\pm1/2}^{m_j}$ form a complete orthonormal set in $L^2(\sigma)^2$, and
\begin{equation}\label{uncert eq1}
(\widetilde\sigma\cdot N)\psi_{j\pm1/2}^{m_j}=\psi_{j\mp1/2}^{m_j}
\quad\text{and}\quad
(1+\widetilde\sigma\cdot L)\psi_{j\pm1/2}^{m_j}=\pm(j+1/2)\psi_{j\pm1/2}^{m_j},
\end{equation}
where $L=-ix\times\nabla$ (see \cite[equation (4.121) and the remark in page 127]{Thaller}).

\begin{lema}\label{sphere l3}
Given $a\in(-m,m)$, there exist positive numbers $d_{j\pm1/2}$ and purely imaginary numbers $p_{j\pm1/2}$ for all $j=1/2,3/2,5/2,\ldots,$ and $m_j=-j,-j+1,\ldots,j$, such that:
\begin{itemize}
\item[$(i)$] $K^a\left(\psi_{j\pm1/2}^{m_j}\right)=d_{j\pm1/2}\,\psi_{j\pm1/2}^{m_j}$ and\,
$\lim_{j\to\infty}d_{j\pm1/2}=0$. Moreover, 
$$0\leq d_{j\pm1/2}\leq d_0=\frac{1-e^{-2\sqrt{m^2-a^2}}}{2\sqrt{m^2-a^2}}.$$
\item[$(ii)$] $W^a\left(\psi_{j\pm1/2}^{m_j}\right)=p_{j\pm1/2}\,\psi_{j\mp1/2}^{m_j}$ and \,$p_{j+1/2}=-p_{j-1/2}$. Moreover,
$$|p_{j\pm1/2}|^2=\frac{1}{4}-(m^2-a^2)d_{j+1/2}\,d_{j-1/2}\geq
\frac{1}{4}\,e^{-2\sqrt{m^2-a^2}}\left(2-e^{-2\sqrt{m^2-a^2}}\right).$$
\end{itemize}
\end{lema}

\begin{proof} For any $n$ and $l$, we identify  the spherical harmonic $Y^l_n$ with its homogeneous extension of degree $0$ to $\R^3\setminus\{0\}$. That is, $Y^l_n(x)=Y^l_n(x/|x|)$ for all $x\in\R^3\setminus\{0\}$. In particular, $|x|^nY_n^l(x)$ is a homogeneous polynomial of degree $n$ which is harmonic. We use the same identification for $\psi^{m_j}_{j\pm1/2}$.

{\em Proof of $(i)$}.
In order to prove the first identity in $(i)$, fix $j$ and $m_j$ and set
\begin{equation}\label{uncert l1 eq11}
u(x)=\int k^a(x-z)\psi^{m_j}_{j\pm1/2}(z)\,d\sigma(z)\quad\text{for $x\in\R^3$.}
\end{equation}

Given $\epsilon>0$, we define $h_\epsilon(x)=\epsilon^{-1}\chi_{(1-\epsilon/2,1+\epsilon/2)}(|x|)$. It is easy to verify that $h_\epsilon\mu$ converges to $\sigma$ in the weak$^*$ topology when $\epsilon\to 0$. In particular, since $\sigma$ and $h_\epsilon\mu$ have compact support and $k^a$ is continuous in $\R^3\setminus\{0\}$ and has exponential decay at infinity, it is not hard to show that actually $u=\lim_{\epsilon\to0}u_\epsilon$ in $L^2(\mu)^2$, where
\begin{equation}\label{uncert l1 eq6}
u_\epsilon(x)=\int k^a(x-z)\psi^{m_j}_{j\pm1/2}(z)h_\epsilon(z)\,d\mu(z)
=k^a*\left(\psi^{m_j}_{j\pm1/2}h_\epsilon\right).
\end{equation}
The term on the right hand side of last equality in (\ref{uncert l1 eq6}) denotes the usual convolution of (matrix and vectors of) functions in $L^2(\mu)$. Applying the Fourier transform to (\ref{uncert l1 eq6}) and using that $k^a$ is a fundamental solution of $(-\Delta+m^2-a^2)I_2$, we obtain
\begin{equation*}
\FF(u_\epsilon)(\xi)
=(4\pi^2|\xi|^2+m^2-a^2)^{-1}\FF\left(\psi^{m_j}_{j\pm1/2}h_\epsilon\right)(\xi).
\end{equation*}
Note that, for any $0<\epsilon<1$, $|x|^{-(j\pm1/2)}h_\epsilon(x)$ is a bounded radial function with compact support, thus \cite[Corollary in page 72]{Stein} shows that 
$$\FF\left(\psi^{m_j}_{j\pm1/2}h_\epsilon\right)(\xi)
=\FF\left(|x|^{j\pm1/2}\psi^{m_j}_{j\pm1/2}(x)\,|x|^{-(j\pm1/2)}h_\epsilon(x)\right)(\xi)
=|\xi|^{j\pm1/2}\psi^{m_j}_{j\pm1/2}(\xi)\,g_\epsilon(\xi)$$ for some radial function $g_\epsilon$ depending on $j\pm1/2$ but not on $m_j$. Hence,
\begin{equation}\label{uncert l1 eq7}
\FF(u_\epsilon)(\xi)
=(4\pi^2|\xi|^2+m^2-a^2)^{-1}g_\epsilon(\xi)\,|\xi|^{j\pm1/2}
\psi^{m_j}_{j\pm1/2}(\xi).
\end{equation}
Since $(4\pi^2|\xi|^2+m^2-a^2)^{-1}g_\epsilon(\xi)$ is also a radial function, we can use that $\FF^4$ is the identity operator and \cite[Corollary of page 72]{Stein} in (\ref{uncert l1 eq7}) to deduce that 
\begin{equation}\label{uncert l1 eq8}
u_\epsilon=f_\epsilon
\psi^{m_j}_{j\pm1/2}
\end{equation}
for some radial function $f_\epsilon$ depending on $j\pm1/2$ but not on $m_j$. 
Finally, using that $u=\lim_{\epsilon\to0}u_\epsilon$ and (\ref{uncert l1 eq8}), we conclude that 
\begin{equation}\label{uncert l1 eq10}
u=f_{j\pm1/2}\psi^{m_j}_{j\pm1/2}
\end{equation} for some radial function $f_{j\pm1/2}$. 
We already know that $K^a$ is a bounded operator in $L^2(\sigma)^2$, and since $k^a(x)=O(1/|x|)I_2$ for $|x|\to0$, one can check that 
\begin{equation}\label{uncert l1 eq9}
K^a\left(\psi^{m_j}_{j\pm1/2}\right)=\Tr_\sigma(u)
=\Tr_\sigma\left(f_{j\pm1/2}\psi^{m_j}_{j\pm1/2}\right),
\end{equation}
which implies that $f_{j\pm1/2}(r)$ is continuous at $r=1$. Then, by setting
$d_{j\pm1/2}=f_{j\pm1/2}(1)$, (\ref{uncert l1 eq9}) shows that $$K^a\left(\psi_{j\pm1/2}^{m_j}\right)
=d_{j\pm1/2}\,\psi_{j\pm1/2}^{m_j}.$$

Concerning the second statement in $(i)$, since $k^a(x)=O(1/|x|)I_2$ for $|x|\to0$, it is easy to check that $K^a$ is a compact operator in $L^2(\sigma)^2$, so the eigenvalues of $K^a$ form a bounded sequence which has $\{0\}$ as the only possible accumulation point (see \cite[Fredholm's Theorem (0.38)$(a)$]{Folland}, for example). Therefore, $\lim_{j\to\infty}d_{j\pm1/2}=0$. 

Let us now prove the last statement in $(i)$. From Lemma \ref{sphere l2}, $K^a$ is a positive operator, which implies that $d_{j\pm1/2}\geq0$ for all $j$. Moreover, 
$d_{j\pm1/2}\leq\|K^a\|_{L^2(\sigma)^2\to L^2(\sigma)^2}$. Following \cite[Generalized Young's Inequality (0.10)]{Folland} and since $S^2$ is invariant under rotations, it is easy to see that 
$$\|K^a\|_{L^2(\sigma)^2\to L^2(\sigma)^2}\leq\|k^a(\cdot-e_3)\|_{L^1(S^2)},$$ where $e_3=(0,0,1)$ (we have identified the matrix $k^a$ with its scalar version). Consider the change of variables to polar coordinates in $S^2$
\begin{equation}\label{polar}
\left\{
\begin{split}
&x=(\sin\varphi\cos\theta,\sin\varphi\sin\theta,\cos\varphi)\quad\text{for }0\leq\theta\leq2\pi,\, 0\leq\varphi\leq\pi,\\
&d\sigma(x)=\sin\varphi\,d\varphi d\theta.
\end{split}
\right.
\end{equation}
Then, we have
\begin{equation*}
\begin{split}
\|k^a(\cdot-e_3)\|_{L^1(S^2)}
&=\int_0^{2\pi}\int_0^\pi\frac{e^{-\sqrt{2(m^2-a^2)(1-\cos\varphi)}}}{4\pi\sqrt{2(1-\cos\varphi)}}\sin\varphi\,d\varphi d\theta
=\frac{1-e^{-2\sqrt{m^2-a^2}}}{2\sqrt{m^2-a^2}},
\end{split}
\end{equation*}
where we used the change of variables $\varphi\to\sqrt{2(1-\cos\varphi)}$ in the last equality above. 

To finish the proof of $(i)$, it only remains to check that $d_0=\|k^a(\cdot-e_3)\|_{L^1(S^2)}$. Note that
\begin{equation}\label{d0 eq}
\psi_0^{1/2}=\left(\begin{array}{c} Y^0_0\\
0 \end{array}\right)
=c\left(\begin{array}{c} 1\\
0 \end{array}\right)\qquad(\text{we take }j=1/2,\,m_j=1/2),
\end{equation}
where $c>0$ is some constant. Therefore,
\begin{equation*}
d_0c\left(\begin{array}{c} 1\\0 \end{array}\right)
=d_0\psi_0^{1/2}(e_3)=K^a\left(\psi_0^{1/2}\right)(e_3)
=c\left(\begin{array}{c} \|k^a(\cdot-e_3)\|_{L^1(S^2)}\\0 \end{array}\right),
\end{equation*}
and we are done.

{\em Proof of $(ii)$}. Fix $j$ and $m_j$. Recall from (\ref{uncert l1 eq11}) and (\ref{uncert l1 eq10}) that, for $x\in\R^3$, we have
\begin{equation}\label{uncert l1 eq1}
u(x)=\int k^a(x-z)\psi^{m_j}_{j\pm1/2}(z)\,d\sigma(z)=f_{j\pm1/2}(r)\,\psi^{m_j}_{j\pm1/2}(\theta)
\end{equation}
for some function $f_{j\pm1/2}\in\CC(0,\infty)$, where $r=|x|$ and $\theta=x/|x|$. Furthermore, $f_{j\pm1/2}\in\CC^\infty((0,1)\cup(1,\infty))$ and, by similar arguments to the ones that prove that $f_{j\pm1/2}(r)$ is continuous at $r=1$, one can show that $\lim_{r\to 1^{+}}f_{j\pm1/2}'(r)$ and $\lim_{r\to 1^{-}}f_{j\pm1/2}'(r)$ exist; we omit the details. 

Since $w^a=-i(\widetilde\sigma\cdot\nabla) k^a$, similarly to (\ref{sphere l2 eq2}) one can check that
\begin{equation}\label{uncert l1 eq2}
2W^a\left(\psi_{j\pm1/2}^{m_j}\right)
=(-i(\widetilde\sigma\cdot\nabla)u)_++(-i(\widetilde\sigma\cdot\nabla)u)_-,
\end{equation}
where $(-i(\widetilde\sigma\cdot\nabla)u)_\pm$ denote the boundary values of $-i(\widetilde\sigma\cdot\nabla)u$ when one approaches to $\Sigma$ non-tangentially from $\Omega_\pm$. It is well known and quite easy to see that 
$\widetilde\sigma\cdot\nabla=(\widetilde\sigma\cdot N)\partial_r
-\frac{1}{r}\,(\widetilde\sigma\cdot N)(\widetilde\sigma\cdot L)$, where $\partial_r=N\cdot\nabla$ and $L=-ix\times\nabla$. Combining this with (\ref{uncert l1 eq2}), (\ref{uncert l1 eq1}) and (\ref{uncert eq1}), we obtain
\begin{equation*}
2W^a\left(\psi_{j\pm1/2}^{m_j}\right)
=-i\left(f_{j\pm1/2}'(1^+)+f_{j\pm1/2}'(1^-)
+2\left(1\mp\left(j+1/2\right)\right)f_{j\pm1/2}(1)\right)\psi_{j\mp1/2}^{m_j},
\end{equation*}
i.e., $W^a\left(\psi_{j\pm1/2}^{m_j}\right)=p_{j\pm1/2}\,\psi_{j\mp1/2}^{m_j}$ for some $p_{j\pm1/2}\in \C$. By $(\ref{uncert eq1})$, Lemmata \ref{sphere l1}$(ii)$ and \ref{sphere l3}$(i)$,
\begin{equation*}
\begin{split}
\left(p_{j\pm1/2}\right)^2\psi_{j\pm1/2}^{m_j}
&=((\widetilde\sigma\cdot N)W^a)^2\left(\psi_{j\pm1/2}^{m_j}\right)\\
&=\left(-1/4+(m^2-a^2)((\widetilde\sigma\cdot N)K^a)^2\right)\left(\psi_{j\pm1/2}^{m_j}\right)\\
&=\left(-1/4+(m^2-a^2)d_{j+1/2}\,d_{j-1/2}\right)
\psi_{j\pm1/2}^{m_j},
\end{split}
\end{equation*}
which yields 
\begin{equation}\label{uncert l1 eq3}
\left(p_{j\pm1/2}\right)^2=-\frac{1}{4}+(m^2-a^2)d_{j+1/2}\,d_{j-1/2}.
\end{equation}
From the last statement in $(i)$, we have
\begin{equation}\label{uncert l1 eq4}
-\frac{1}{4}+(m^2-a^2)d_{j+1/2}\,d_{j-1/2}
\leq\frac{1}{4}\left(\left(1-e^{-2\sqrt{m^2-a^2}}\right)^2-1\right)<0,
\end{equation}
thus $\left(p_{j\pm1/2}\right)^2<0$ by (\ref{uncert l1 eq3}), and that means that $p_{j\pm1/2}$ are purely imaginary numbers. The last statement in $(ii)$ follows by (\ref{uncert l1 eq3}) and (\ref{uncert l1 eq4}), so it only remains to prove that $p_{j+1/2}=-p_{j-1/2}$. For that purpose, we use the first identity in $(ii)$ and that $\widetilde\sigma\cdot N$ and $W^a$ are symmetric operators to see that
\begin{equation}\label{uncert l1 eq5}
\begin{split}
2\Re(ip_{j\pm1/2})\|\psi_{j\mp1/2}^{m_j}\|^2_{L^2(S^2)^2}
&=2\Re\left(\int_{S^2}iW^a\left(\psi_{j\pm1/2}^{m_j}\right)
\cdot\overline{(\widetilde\sigma\cdot N)\psi_{j\pm1/2}^{m_j}}\,d\sigma\right)\\
&=\int_{S^2}i\left((\widetilde\sigma\cdot N)W^a-W^a(\widetilde\sigma\cdot N)\right)
\left(\psi_{j\pm1/2}^{m_j}\right)
\cdot\overline{\psi_{j\pm1/2}^{m_j}}\,d\sigma\\
&=i(p_{j\pm1/2}-p_{j\mp1/2})\|\psi_{j\pm1/2}^{m_j}\|^2_{L^2(S^2)^2}.
\end{split}
\end{equation}
Since $\|\psi_{j\pm1/2}^{m_j}\|_{L^2(S^2)^2}=1$ and we already know that $p_{j\pm1/2}$ are purely imaginary, we obtain from (\ref{uncert l1 eq5}) that
$2ip_{j\pm1/2}=2\Re(ip_{j\pm1/2})=i(p_{j\pm1/2}-p_{j\mp1/2})$, which implies that $p_{j+1/2}=-p_{j-1/2}$. The lemma is finally proved.

Note that, if we know that $d_{j\pm1/2}\neq0$ for all $j$, the relation $p_{j+1/2}=-p_{j-1/2}$ also follows from Lemma \ref{sphere l1}$(i)$.
\end{proof}

The following theorem is based on an uncertainty principle on the sphere and goes on the lines of \cite[Theorem 1]{DDEV}.

\begin{teo}\label{thmineq}
Given $\lambda>0$ and $a\in(-m,m)$, the operator $1/\lambda+(m+a)K^a$ is invertible in $L^2(\sigma)^2$. Furthermore, for any $ f\in L^2(\sigma)^2$ and any $\delta>0$, we have
\begin{equation}\label{ineq}
\begin{split}
\int_{S^2} | f|^2\,d\sigma
&\leq \frac{1}{2M\delta}\int_{S^2} \left(1/\lambda+(m+a)K^a\right)^{-1}(W^a( f))\cdot\overline{W^a(f)}\,d\sigma\\ 
&\quad+\frac{\delta}{2M}\int_{S^2} \left(1/\lambda+(m+a)K^a\right)((\widetilde\sigma\cdot N)f)\cdot\overline{(\widetilde\sigma\cdot N) f}\,d\sigma,
\end{split}
\end{equation}
where $M=\min_{j}|p_{j\pm1/2}|\geq \frac{1}{2}\,e^{-\sqrt{m^2-a^2}}
\sqrt{2-e^{-2\sqrt{m^2-a^2}}}.$
If $j_0$ is such that $M=|p_{j_0\pm1/2}|$, then equality in $(\ref{ineq})$ holds by taking $f=\psi_{j_0\pm1/2}^{m_{j_0}}$ for any $m_{j_0}=-j_0,-j_0+1,\ldots,j_0$ and 
\begin{equation}\label{delta}
\delta=(1/\lambda+(m+a)d_{j_0\mp1/2})^{-1}M.
\end{equation}
\end{teo}

\begin{proof}
Recall that $K^a$ is a positive operator by Lemma \ref{sphere l2}, thus  $1/\lambda+(m+a)K^a$ is positive and invertible in $L^2(\sigma)^2$ for all $\lambda>0$ by \cite[Theorem 12.12$(c)$]{Rudin}, and the inverse is also a positive operator.

Since $d_{j\pm1/2}\geq0$ for all $j$ and $\lim_{j\to\infty}d_{j\pm1/2}=0$ by Lemma \ref{sphere l3}$(i)$, Lemma \ref{sphere l3}$(ii)$ shows that $|p_{j\pm1/2}|\leq1/2$ for all $j$ and 
$\lim_{j\to\infty}|p_{j\pm1/2}|=1/2$. Therefore, there exists some $j_0$ such that $|p_{j_0\pm1/2}|=\inf_{j}|p_{j\pm1/2}|$, and thus $M$ is well defined. The estimate of $M$ from below is stated in Lemma \ref{sphere l3}$(ii)$.

Given $j$ and $m_j$, we define
\begin{equation}\label{eq1}
A=\left|2\Re\left( \int_{S^2}iW^a\left(\psi_{j\pm1/2}^{m_j}\right)\cdot\overline{(\widetilde\sigma\cdot N) \psi_{j\pm1/2}^{m_j}}\,d\sigma\right)\right|.
\end{equation}
We will prove (\ref{ineq}) by estimating $A$ from above and from below.

From Lemma \ref{sphere l3} we know that $p_{j+1/2}=-p_{j-1/2}\neq 0$ are purely imaginary. Therefore, arguing as in (\ref{uncert l1 eq5}), we have
\begin{equation}\label{eq2}
\begin{split}
A=2|p_{j\pm1/2}|\|\psi_{j\pm1/2}^{m_j}\|^2_{L^2(S^2)^2}
\geq 2M\|\psi_{j\pm1/2}^{m_j}\|^2_{L^2(S^2)^2}.
\end{split}
\end{equation}

To estimate $A$ from above, we use Lemma \ref{sphere l3}, (\ref{uncert eq1}), Cauchy-Schwarz inequality, and that $2xy\leq x^2+y^2$ for all $x,y\geq0$, to deduce that
\begin{equation}\label{eq3}
\begin{split}
A&=\left|2\Re \left(\int_{S^2}ip_{j\pm1/2}\,\psi_{j\mp1/2}^{m_j}\cdot\overline{ \psi_{j\mp1/2}^{m_j}}\,d\sigma\right)\right|\\ 
&\leq \int_{S^2}\big(\delta(1/\lambda+(m+a)d_{j\mp 1/2})\big)^{-1}
\left|p_{j\pm1/2}\,\psi_{j\mp1/2}^{m_j}\right|^2\,d\sigma\\
&\quad+ \int_{S^2}\delta(1/\lambda+(m+a)d_{j\mp 1/2})\left|\psi_{j\mp1/2}^{m_j}\right|^2\,d\sigma\\ 
&=\frac{1}{\delta}\int_{S^2}\left(\left(1/\lambda+(m+a)K^a\right)^{-1}W^a\right)\left(\psi_{j\pm1/2}^{m_j}\right)\cdot
\overline{W^a\left(\psi_{j\pm1/2}^{m_j}\right)}\,d\sigma\\
&\quad+\delta\int_{S^2}\Big((1/\lambda+(m+a)K^a)(\widetilde\sigma\cdot N)\Big)\left(\psi_{j\pm1/2}^{m_j}\right)\cdot\overline{(\widetilde\sigma\cdot N)\psi_{j\pm1/2}^{m_j}}\,d\sigma.
\end{split}
\end{equation}

From (\ref{eq2}) and (\ref{eq3}), we see that (\ref{ineq}) holds for $f=\psi_{j\pm1/2}^{m_j}$. The functions $\psi_{j\pm1/2}^{m_j}$ with $j=1/2,3/2,5/2,\ldots,$ and $m_j=-j,-j+1,\ldots,j$ form a complete orthonormal system in $L^2(\sigma)^2$. Hence, to prove (\ref{ineq}) in full generality, we first write any $f\in L^2(\sigma)^2$ as a linear combination of the $\psi_{j\pm1/2}^{m_j}$'s and we expand the left and right hand side of (\ref{ineq}) in terms of this basis. Then, using the orthogonality and, for the right hand side of (\ref{ineq}), that
$$K^a\left(\psi_{j\pm1/2}^{m_j}\right)=d_{j\pm1/2}\,\psi_{j\pm1/2}^{m_j}\quad\text{and}\quad
W^a\left(\psi_{j\pm1/2}^{m_j}\right)=p_{j\pm1/2}\,\psi_{j\mp1/2}^{m_j}$$
from Lemma \ref{sphere l3}, we conclude that (\ref{ineq}) holds for all $f\in L^2(\sigma)^2$ if and only if it holds for all $\psi_{j\pm1/2}^{m_j}$. Therefore, (\ref{ineq}) is finally proved.

It only remains to check the last statement of the theorem. Let $j_0$ be such that $M=|p_{j_0\pm1/2}|$ (we already know that such $j_0$ exists).  Then, for the functions
$$\psi_{j_0\pm1/2}^{m_{j_0}}\quad\text{for any }m_{j_0}=-j_0,-j_0+1,\ldots,j_0,$$ the inequality in (\ref{eq2}) becomes an equality. Furthermore, for $\delta$ satisfying (\ref{delta}), we have
$$\frac{|p_{j_0\pm1/2}|}{\sqrt{\delta
(1/\lambda+(m+a)d_{j_0\mp 1/2})}}\,\psi_{j_0\mp1/2}^{m_{j_0}}
= \sqrt{\delta(1/\lambda+(m+a)d_{j_0\mp 1/2})}\psi_{j_0\mp1/2}^{m_{j_0}},$$
which implies that the inequality in (\ref{eq3}) is an equality for $\psi_{j_0\pm1/2}^{m_{j_0}}$.
\end{proof}

Recall from Lemma \ref{sphere l3}$(ii)$ that 
$|p_{j+1/2}|=|p_{j-1/2}|$ for all $j$. Hence, for any $j_0$ such that $M=|p_{j_0\pm1/2}|$, we have two possible elections of the subindex, say $j_0+1/2$ and $j_0-1/2$, and therefore two possible values of $\delta$ for which equality  in (\ref{ineq}) holds. Hence, we get two (a priori different) sharp inequalities. The same observation applies if such $j_0$ is not unique.

Theorem \ref{thmineq} has an interesting consequence concerning a lower bound for the $2$-dimensional Riesz transform on the sphere. Given a finite Borel measure $\nu$ in $\R^3$, $h\in L^2(\nu)$ and $x\in\R^3$, one defines the $2$-dimensional Riesz transform of $h$ as
$$R_\nu(h)(x)=\lim_{\epsilon\searrow0}\int_{|x-y|>\epsilon}
\frac{x-y}{|x-y|^3}\,h(y)\,d\nu(y),$$
whenever the limit makes sense. It is well known that $R_\nu: L^2(\nu)\to L^2(\nu)^3$ is a bounded operator for 2-dimensional uniformly rectifiable AD regular measures $\nu$ in $\R^3$ (see \cite{DS2} for a deep study on this subject). In particular, $R_\nu: L^2(\nu)\to L^2(\nu)^3$ is a bounded operator when $\nu$ is the surface measure of a bounded Lipschitz domain. This means that 
$$\|R_\nu(h)\|_{L^2(\nu)^3}\leq C\|h\|_{L^2(\nu)}$$ for some constant $C>0$ and all $h\in L^2(\nu)$. 
Much less is said about lower $L^2$-bounds for the Riesz transform. However, in the case of the sphere, from the results in \cite{HMMPT} (see also \cite[equation (4.6.9)]{HMT}) one can easily show that the Riesz transform (multiplied by a suitable constant) is an isometry, providing sharp constants to the above-mentioned inequalities. The following corollary of Theorem \ref{thmineq} yields the sharp constant of the inequality from below on $S^2$.

Set $w(x)=(4\pi|x|^3)^{-1}i\,\widetilde{\sigma}\cdot x$ for $x\in\R^3\setminus\{0\}$, and
\begin{equation*}
W(f)(x)=\lim_{\epsilon\searrow0}\int_{|x-z|>\epsilon} w(x-z)f(z)\,d\sigma(z)\quad\text{for $x\in S^2$ and $f\in L^2(\sigma)^2$.}
\end{equation*}

\begin{coro}\label{riesz}
The following inequalities hold and they are sharp:
\begin{itemize}
\item[$(i)$] $\|f\|_{L^2(\sigma)^2}\leq 2\|W(f)\|_{L^2(\sigma)^2}$ for all $f\in L^2(\sigma)^2$.
\item[$(ii)$] $2\pi\|h\|_{L^2(\sigma)}\leq\|R_\sigma(h)\|_{L^2(\sigma)^3}$ for all real-valued $h\in L^2(\sigma)$.
\end{itemize}
\end{coro}

\begin{proof}
Set $a=0$ and $\lambda=1$ in Theorem \ref{thmineq}. Given $f\in L^2(\sigma)^2$, (\ref{ineq}) yields
\begin{equation}\label{sphere c1 eq1}
\begin{split}
\int_{S^2} |f|^2\,d\sigma
&\leq \frac{1}{2M\delta}\int_{S^2} \left(1+mK^0\right)^{-1}
(W^0( f))\cdot\overline{W^0(f)}\,d\sigma\\ 
&\quad+\frac{\delta}{2M}\int_{S^2} \left(1+mK^0\right)((\widetilde\sigma\cdot N)f)\cdot\overline{(\widetilde\sigma\cdot N) f}\,d\sigma,
\end{split}
\end{equation}
where $M=\min_j|p_{j\pm1/2}|$.
Notice that, by Lemma \ref{sphere l3}, $d_{j\pm1/2}$ are uniformly bounded and  $|p_{j\pm 1/2}|\to1/2$ uniformly in $j$ when $m\to0$.
In particular, $M\to1/2$ when $m\to0$. Recall that $K^0$ and $W^0$ are defiend by means of the convolution kernels
\begin{equation*}
k^0(x)=\frac{e^{-m|x|}}{4\pi|x|}\,I_2
\quad\text{and}\quad
w^0(x)=\frac{e^{-m|x|}}{4\pi|x|^3}
\left(1+m|x|\right)\,i\,\widetilde{\sigma}\cdot x.
\end{equation*}
We define
$k(x)=(4\pi|x|)^{-1}I_2$ for $x\in\R^3\setminus\{0\}$, and
\begin{equation*}
K(f)(x)=\int k(x-z)f(z)\,d\sigma(z)\quad\text{for $x\in S^2$.} 
\end{equation*}
That $K$ and $W$ are bounded operators in $L^2(\sigma)^2$ follows essentially as in the case of $K^0$ and $W^0$. Moreover, it is not hard to show that 
\begin{equation}\label{pirris}
\begin{split}
\|K-K^0\|_{L^2(\sigma)^2\to L^2(\sigma)^2}&=O(m)\quad\text{for $m\to0$,}\\
\|W-W^0\|_{L^2(\sigma)^2\to L^2(\sigma)^2}&=O(m)\quad\text{for $m\to0$.}
\end{split}
\end{equation}
(see the proof of Lemma \ref{spec l1} for a related argument). Roughly speaking, in $S^2$, $K^0$ and $W^0$ are compact perturbations of $K$ and $W$ which depend on $m$ continuously. Therefore, if we take $m\to0$ in (\ref{sphere c1 eq1}) and we use that $(\widetilde\sigma\cdot N)^2=I_2$, we obtain
$$\int_{S^2}|f|^2\,d\sigma\leq\frac{1}{\delta}\int_{S^2}|W(f)|^2\,d\sigma+\delta\int_{S^2}|f|^2\,d\sigma$$
for all $\delta>0$.
Minimazing in $\delta$, i.e., taking $\delta=\|W(f)\|_{L^2(\sigma)^2}/\|f\|_{L^2(\sigma)^2}$, we get
\begin{equation}\label{sphere c1 eq2}
\|f\|_{L^2(\sigma)^2}\leq 2\|W(f)\|_{L^2(\sigma)^2},
\end{equation}
which is the inequality in Corollary \ref{riesz}$(i)$.
Corollary \ref{riesz}$(ii)$ follows from (\ref{sphere c1 eq2}) by taking 
$f=\binom{h}{0}$ for any real-valued $h\in L^2(\sigma)$.
That the inequalities are sharp is a consequence of the fact that (\ref{sphere c1 eq1}) is sharp for $\delta$ as in (\ref{delta}). Since $M\to1/2$ and $\delta\to1/2$ for $m\to0$ (recall that we have set $a=0$ and $\lambda=1$), using (\ref{sphere c1 eq1}) and (\ref{pirris}) one can check that (\ref{sphere c1 eq2}) is sharp, and the corollary follows. 
\end{proof}

The following lemma gives a specific criterion based on Proposition \ref{spec p1} to generate eigenvectors of $H+V_\lambda$.
\begin{lema}\label{eigenfunction}
Let $H+V_{\lambda}$ be as in $(\ref{domain})$ with $\Sigma=S^2$.
If $\lambda>0$ and $a\in(-m,m)$ satisfy
\begin{equation}\label{a}
\lambda^2/4-\left((m+a)d_{j\mp 1/2}-(m-a)d_{j\pm 1/2}\right)\lambda=1\quad\text{for some $j$},
\end{equation} 
then, for any $m_j$, $\psi_{j\pm1/2}^{m_j}$ gives rise to an eigenfunction of $H+V_{\lambda}$ with eigenvalue $a$.
\end{lema}

\begin{proof}
Let $\lambda$, $a$ and $j$ be as in the lemma. Since $\lambda>0$, $1/\lambda+(a+m)K^a$ is invertible by Theorem \ref{thmineq}. Hence we can define
$$g=\left( \begin{array}{rr}  f  \\  h  \end{array} \right)\in L^2(\sigma)^4,\quad\text{where } h=\psi_{j\pm1/2}^{m_{j}}\text{ and }
f=-\left((1/\lambda+(a+m)K^a)^{-1}W^a\right)( h).$$
In particular, we have the relation
\begin{equation}\label{eigenfunction eq2}
(a+m)K^a( f)+W^a( h)=-\frac{1}{\lambda} f.
\end{equation}

Using Lemma \ref{sphere l3}, we have
\begin{equation}\label{eigenfunction eq1}
\begin{split}
-W^a(f)+(m-a)K^a(h)
&=\left(W^a(1/\lambda+(a+m)K^a)^{-1}W^a\right)( h)+(m-a)K^a( h)\\
&=\left(|p_{j\pm1/2}|^2
(1/\lambda+(a+m)d_{j\mp1/2})^{-1}+(m-a)d_{j\pm1/2}\right)h\\
&=\left(\frac{1/4-(m^2-a^2)d_{j+1/2}\,d_{j-1/2}}
{1/\lambda+(a+m)d_{j\mp1/2}}+(m-a)d_{j\pm1/2}\right)h.
\end{split}
\end{equation}
If $\lambda>0$ and $a\in(-m,m)$ satisfy (\ref{a}), then 
$$\frac{1/4-(m^2-a^2)d_{j+1/2}\,d_{j-1/2}}
{1/\lambda+(a+m)d_{j\mp1/2}}+(m-a)d_{j\pm1/2}=\frac{1}{\lambda},$$
which by (\ref{eigenfunction eq1}) implies that 
\begin{equation}\label{eigenfunction eq3}
W^a(f)+(a-m)K^a(h)=-\frac{1}{\lambda}h.
\end{equation}
Finally, combining (\ref{sphere eq1}), (\ref{eigenfunction eq2}) and (\ref{eigenfunction eq3}), we obtain
\begin{equation*}
\left.C_\sigma^a(g)=
\left( \begin{array}{rr}  
(a+m)K^a( f)+W^a( h)\\
W^a( f)+(a-m)K^a( h)
\end{array} \right)
=-\frac{1}{\lambda}\,g,
\right.
\end{equation*}
which means that  $g\in\Ker(1/\lambda+C_\sigma^a)$. Following Proposition \ref{spec p1} (see also Theorem \ref{spec t1}), if we set $G=a\Phi^a(g)$ and $\varphi=\Phi(G+g)$, then $(H+V_{\lambda})(\varphi)=a\varphi$. The lemma is proved.
\end{proof}

For the case $\lambda<0$, one can develop results analogous to Theorem \ref{thmineq} and Lemma \ref{eigenfunction} by repeating the arguments involved in the proofs but by using the invertibility of $1/\lambda-(m-a)K^a$ instead of $1/\lambda+(a+m)K^a$. We leave the details for the reader.

\subsection{Further comments}\label{further com}
 \subsubsection{Existence of eigenfunctions}
From Lemma \ref{spec l1} we have that $C=\sup_{a\in(-m,m)}\|C_\sigma^a\|<\infty$, and in Theorem \ref{spec t1} we proved that if 
\begin{equation}\label{pepa2}
|\lambda|\not\in\left[C^{-1},
4C\right]
\end{equation} 
then $H+V_\lambda$ has no eigenvalues in $(-m,m)$ (recall that $C\geq1/2$ by (\ref{spec t1 eq1})). Furthermore, in Remark \ref{lambda-a} we showed that, for any $a\in(-m,m)$, there exists some $\lambda$ such that $H+V_\lambda$ has $a$ as an eigenvalue. From these results, it is not clear what can be said positively about the set of $\lambda$'s for which there exist an eigenvalue of $H+V_\lambda$. Thanks to Lemma \ref{eigenfunction}, we can give a bit more of information in the case of the sphere, but first we need to do some computations. 

Lemma \ref{sphere l3}$(i)$ gives a precise value for $d_0$ in terms of $m>0$ and $a\in(-m,m)$, that is,
\begin{equation}\label{d0bis eq}
d_0=\frac{1-e^{-2\sqrt{m^2-a^2}}}{2\sqrt{m^2-a^2}}.
\end{equation}
Following a similar argument, we are going to prove that
\begin{equation}\label{d1 eq}
d_1=\frac{1}{2\sqrt{m^2-a^2}}\left(1-\frac{1}{m^2-a^2}+\left(1+\frac{1}{\sqrt{m^2-a^2}}\right)^2e^{-2\sqrt{m^2-a^2}}\right),
\end{equation}
where $d_1$ corresponds to $d_{j+1/2}$ with $j=1/2$. From (\ref{uncert eq1}) and (\ref{d0 eq}), we have 
\begin{equation*}
\psi^{1/2}_1(x)=\left(\widetilde\sigma\cdot x\right)\psi^{1/2}_0
=c\left(\widetilde\sigma\cdot x\right)\left(\begin{array}{c} 1\\
0 \end{array}\right)= c \left(\begin{array}{c} x_3\\
x_1+ix_2 \end{array}\right).
\end{equation*}
Therefore, if we identify the matrix $k^a$ with its scalar version and we set $e_3=(0,0,1)$, Lemma \ref{sphere l3}$(i)$ yields
\begin{equation}\label{d1 eq1}
d_1c\left(\begin{array}{c} 1\\0 \end{array}\right)
=d_1\psi_1^{1/2}(e_3)=K^a\left(\psi_1^{1/2}\right)(e_3)
=c\left(\begin{array}{c} \int k^a(x-e_3)x_3\,d\sigma(x)\\
\int k^a(x-e_3)(x_1+ix_2)\,d\sigma(x) \end{array}\right).
\end{equation}
That $\int k^a(x-e_3)(x_1+ix_2)\,d\sigma(x)=0$ follows from (\ref{d1 eq1}), but it can be also verified using the change of variables (\ref{polar}) and noting that the resulting integrals contain a $\cos \theta$ or a $\sin\theta$, which integrated on $[0, 2\pi]$ vanish. On the other hand, using (\ref{polar}),
\begin{equation}\label{d1 eq2}
\begin{split}
\int k^a(x-e_3)x_3&\,d\sigma(x)
=\frac{1}{2}\int_0^\pi\frac{e^{-\sqrt{2(m^2-a^2)(1-\cos\varphi)}}}{\sqrt{2(1-\cos\varphi)}}\cos\varphi\sin\varphi\,d\varphi\\
&= \frac{1}{2\sqrt{m^2-a^2}}\left(1-\frac{1}{m^2-a^2}+\left(1+\frac{1}{\sqrt{m^2-a^2}}\right)^2e^{-2\sqrt{m^2-a^2}}\right),
\end{split}
\end{equation}
where we used the change of variables $\varphi\to\sqrt{2(1-\cos\varphi)}$ and integration by parts in the last equality above. Combining (\ref{d1 eq1}) and (\ref{d1 eq2}), we get (\ref{d1 eq}), as desired. Recall that Lemma \ref{sphere l3}$(i)$ states that $d_{j\pm1/2}\geq0$ for al $j$. It is an exercise to check directly from (\ref{d1 eq}) that $d_1\geq0$ for all $a\in(-m,m)$.

We now turn to Lemma \ref{eigenfunction}, which can be used to provide eigenfunctions of $H+V_\lambda$ under some assumptions on $\lambda>0$ and $a\in(-m,m)$. For $j=1/2$, (\ref{a}) reads as
\begin{eqnarray}
\lambda^2/4-\left((m+a)d_{0}-(m-a)d_1\right)\lambda=1,\label{pepaa}\\
\lambda^2/4-\left((m+a)d_1-(m-a)d_0\right)\lambda=1.\label{pepab}
\end{eqnarray}
Using (\ref{d0bis eq}) and (\ref{d1 eq}), it is not difficult to see that if $\lambda$ is very big or very small, then  (\ref{pepaa}) and (\ref{pepab}) do not hold for any $a\in(-m,m)$, so Lemma \ref{eigenfunction} can not be used. This agrees with the above-mentioned result on non-existence of eigenvalues given by (\ref{pepa2}). 

Let us take $m=1$ for simplicity. Figure \ref{fig1} shows the set of points $(a,\lambda)\in(-1,1)\times[0,\infty)$ such that (\ref{pepaa}) and (\ref{pepab}) hold. We see that for any $a\in(-1,1)$ we can take a $\lambda$ for which (\ref{pepaa}) holds, hence there exists an eigenfunction with eigenvalue $a$, by Lemma \ref{eigenfunction}. This agrees with Remark \ref{lambda-a}. 

However, from Figure \ref{fig1} we also see that for any $\lambda$ in some interval there exists an $a\in(-1,1)$ such that (\ref{pepaa}) or (\ref{pepab}) hold. More precisely, using the computer one can show that
if $\lambda\in(-4/3+\sqrt{2}\sqrt{26}/3,4+\sqrt{10}\sqrt{2})$ then there exists $a\in(-1,1)$ such that (\ref{pepaa}) holds, and if $\lambda\in(-4+\sqrt{10}\sqrt{2},4/3+\sqrt{2}\sqrt{26}/3)$ then there exists $a\in(-1,1)$ such that (\ref{pepab}) holds. In particular, if $\lambda\in(-4+\sqrt{10}\sqrt{2},4+\sqrt{10}\sqrt{2})$ then there exists $a\in(-1,1)$ such that either (\ref{pepaa}) or (\ref{pepab}) hold. By Lemma \ref{eigenfunction}, this means that the set of $\lambda$'s for which there exists an eigenfunction of $H+V_\lambda$ contains the interval $(-4+\sqrt{10}\sqrt{2},4+\sqrt{10}\sqrt{2})$ (recall that we have set $m=1$ for these calculations). 
\begin{figure}[ht]
\begin{center}
\scalebox{0.34}{\includegraphics{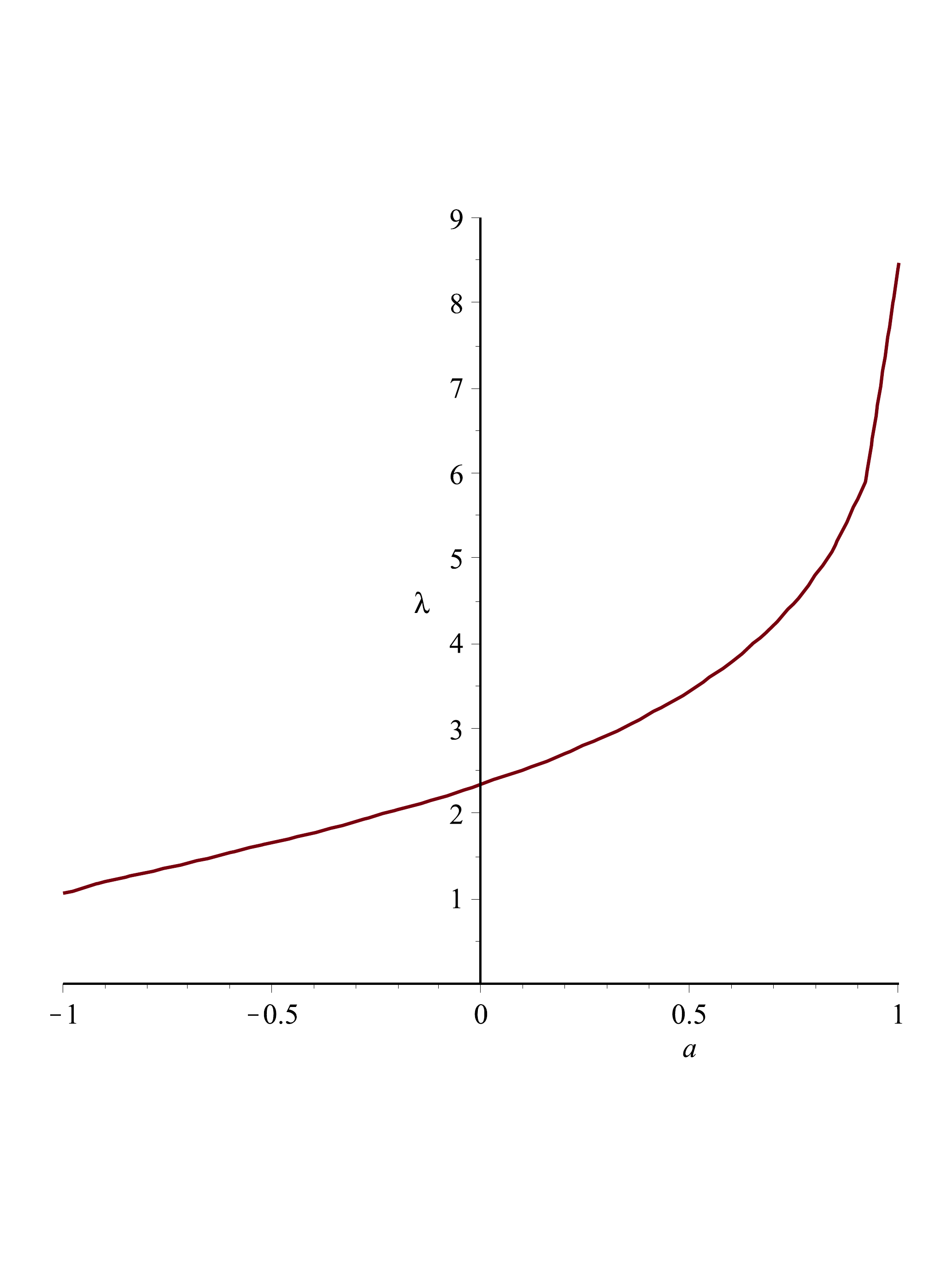}}
\scalebox{0.34}{\includegraphics{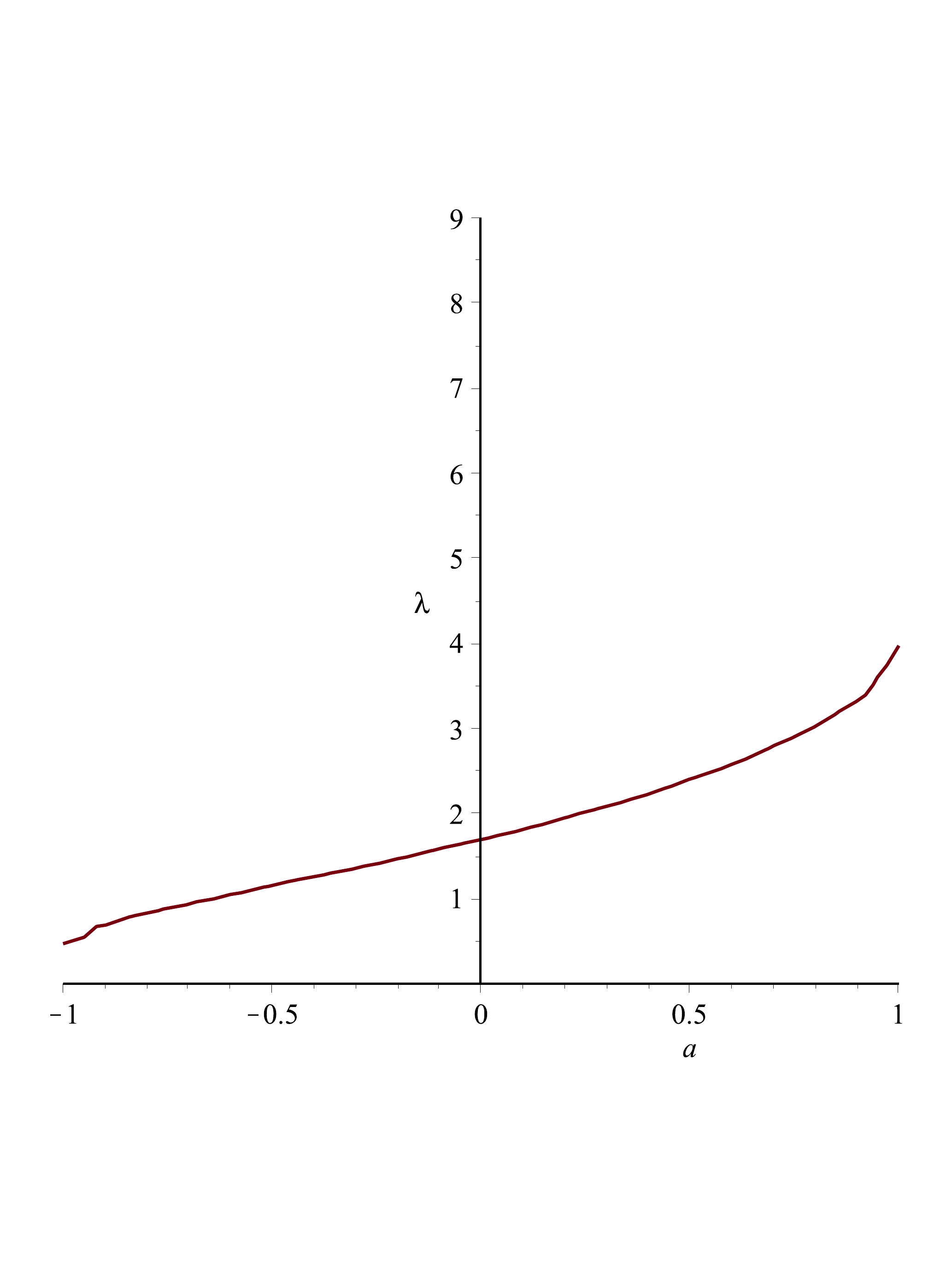}}
\caption{\label{fig1} The curve on the left shows the set of points $$(a,\lambda)\in(-1,1)\times[0,\infty)$$ such that (\ref{pepaa}) holds, and the curve on the right corresponds to (\ref{pepab}). 
The range of admissible $\lambda$'s for the curve on the left is $$(-4/3+\sqrt{2}\sqrt{26}/3,4+\sqrt{10}\sqrt{2})
\approx(1.0703,8.4721),$$ and the range of $\lambda$'s for the curve on the right  is $$(-4+\sqrt{10}\sqrt{2},4/3+\sqrt{2}\sqrt{26}/3)
\approx(0.4721,3.7370).$$}
\end{center}
\end{figure}

\subsubsection{Minimizers and eigenfunctions}\label{delib}
The combination of Theorem \ref{thmineq} and Lemma \ref{eigenfunction} yields an interesting relation between the minimizing functions of (\ref{ineq}) and some eigenfunctions of $H+V_\lambda$. More precisely, let $j_0$ be such that 
$M=\min_j|p_{j\pm1/2}|=|p_{j_0\pm1/2}|$ (recall that $|p_{j+1/2}|=|p_{j-1/2}|$ for all $j$).
Assume that the given $\lambda>0$ and $a\in(-m,m)$ satisfy 
\begin{equation}\label{a1}
\lambda^2/4-\left((m+a)d_{j_0-1/2}-(m-a)d_{j_0+ 1/2}\right)\lambda=1,
\end{equation}
i.e., the election of the first sign in (\ref{a}) holds for this particular $j_0$. Then Lemma \ref{eigenfunction} shows that 
$\psi_{j_0+1/2}^{m_{j_0}}$ gives rise to an eigenfunction of $H+V_{\lambda}$ with eigenvalue $a$, for any $m_{j_0}$.

Fix $\delta$ as in (\ref{delta}), but we chose the first sign on the possible definitions of $\delta$, i.e.,
\begin{equation*}\label{delta1}
\delta=(1/\lambda+(m+a)d_{j_0-1/2})^{-1}M.
\end{equation*}
Then the functions $\psi_{j_0+1/2}^{m_{j_0}}$ are minimizers of (\ref{ineq}), that is, they attain the equality in (\ref{ineq}).

Once $\lambda$, $a$ and $\delta$ are fixed depending on $j_0$, let $J$ be set of $j$'s such that $M=|p_{j\pm1/2}|$ and that one sign election in (\ref{delta}) is satisfied for $j$. In particular, $j_0\in J$. Given $j\in J$, since 
$0<M=|p_{j_0\pm1/2}|=|p_{j\pm1/2}|$ and both $j_0$ and $j$ satisfy (\ref{delta}), we easily deduce that either 
\begin{equation}\label{cases}
d_{j_0\pm1/2}=d_{j\pm1/2}\quad\text{or}\quad
d_{j_0\pm1/2}=d_{j\mp1/2}.
\end{equation}
An inspection of the proof of Theorem \ref{thmineq} shows that, 
\begin{itemize}
\item if $d_{j_0\pm1/2}=d_{j\pm1/2}$ then $\psi_{j+1/2}^{m_{j}}$ are minimizers of (\ref{ineq}), and
\item if $d_{j_0\pm1/2}=d_{j\mp1/2}$, then $\psi_{j-1/2}^{m_{j}}$ are minimizers of (\ref{ineq}).
\end{itemize}
Actually, because of the orthogonality, it can be seen that any minimizer of (\ref{ineq}) must be a linear combination of these functions indexed by $j\in J$ and by a choose of the sign in $j\pm1/2$ depending on (\ref{cases}).
Similarly, for any $j\in J$, (\ref{cases}) and (\ref{a1}) show that $(\ref{a})$ holds for a suitable election of the sign in $j\pm1/2$, and the corresponding functions $\psi_{j\pm1/2}^{m_{j}}$ give rise to eigenfunctions of $H+V_{\lambda}$ with eigenvalue $a$.

Combining the above-mentioned arguments we can conclude that, once $\lambda$, $a$, and $\delta$ are properly chosen, then any function which attains the equality in (\ref{ineq}) give rise to an eigenfunction of $H+V_{\lambda}$ with eigenvalue $a$. Roughly speaking, the minimizers of (\ref{ineq}) provide eigenfunctions of $H+V_{\lambda}$.

\subsubsection{An open question and consequences of a positive answer}
From Lemma \ref{sphere l3}(i), we know that $0\leq d_{j\pm1/2}\leq d_0$ for all $j=1/2, 3/2, 5/2, \ldots$ and $\lim_{j\to\infty}d_{j\pm1/2}=0$, so it is to be expected that the following question has a positive answer. 
\begin{ques}\label{ques}
Is it true that $d_{j+1/2}d_{j-1/2}<d_1d_0$ for all $j=3/2, 5/2, 7/2\ldots$?
\end{ques}
If so, from Lemma \ref{sphere l3}$(ii)$ we see that the minimum in the definition of $M$ in (\ref{ineq}) would be attained only at one particular $j$, namely $j=1/2$. Actually,  $M$ could be calculated explicitely using (\ref{d0bis eq}) and (\ref{d1 eq}), that is,
\begin{equation}\label{M}
M=\frac{1}{2\sqrt{m^2-a^2}}
-\frac{1}{2}\left(1+\frac{1}{\sqrt{m^2-a^2}}\right)
e^{-2\sqrt{m^2-a^2}}.
\end{equation}
The same could be said about the two possible values of $\delta$ in (\ref{delta}), say
$$\delta_0=(1/\lambda +(m+a)d_0)^{-1}M
\quad\text{and}\quad \delta_1=(1/\lambda +(m+a)d_1)^{-1}M.$$

If Question \ref{ques} has a positive answer, the argument of Section \ref{delib} becomes much more transparent, since in this case $J=\{1/2\}$. Furthermore, it would yield the following result: {\em let $a\in(-m,m)$ and $\lambda>0$.
Then, for any $f\in L^2(\sigma)^2$, 
\begin{equation}\label{ineqbis}
\begin{split}
\int_{S^2} | f|^2\,d\sigma
&\leq \frac{1/\lambda +(m+a)d_0}{2M^2}\int_{S^2} \left(1/\lambda+(m+a)K^a\right)^{-1}
(W^a( f))\cdot\overline{W^a(f)}\,d\sigma\\ 
&\quad+\frac{1}{2(1/\lambda +(m+a)d_0)}\int_{S^2} \left(1/\lambda+(m+a)K^a\right)((\widetilde\sigma\cdot N)f)
\cdot\overline{(\widetilde\sigma\cdot N) f}\,d\sigma,
\end{split}
\end{equation}
where $M$ is given by $(\ref{M})$.
The equality in $(\ref{ineqbis})$ is only attained at linear combinations of $\psi_1^{l}$ for $l\in\{-1/2,1/2\}$. Moreover, if
\begin{equation}\label{a2}
\frac{\lambda^2}{4}-\left((m+a)d_{0}-(m-a)d_{1}\right)\lambda=1
\quad\text{(see $(\ref{d0bis eq})$ and $(\ref{d1 eq})$)},
\end{equation}
then the minimizers of $(\ref{ineqbis})$ give rise to eigenfunctions of $H+V_\lambda$. 

These conclusions also hold if we exchange the roles of $d_0$ and $d_1$ in $(\ref{ineqbis})$ and $(\ref{a2})$ and we replace $\psi_1^{l}$ by $\psi_0^{l}$ (that is, we exchange the roles of $j+1/2$ and $j-1/2$ for $j=1/2$).}

\section{On the confinement}\label{s confinement}
In this section, we show a criterion on $H+V$ to generate confinement, namely Theorem \ref{conf t1}. This criterion is stated in terms of an algebraic property of certain bounded operators in $L^2(\sigma)^4$. An application to electrostatic and Lorentz scalar shell potentials is also shown. But before, we need some auxiliary lemmata.
\begin{lema}\label{conf l1}
Let $T$ be as in Theorem \ref{pre t1}. Then, $\chi_{\Omega_\pm}\varphi\in D(T)$ for all $\varphi\in D(T)$ if
\begin{equation}\label{conf l1 eq3}
\{C_\sigma(\alpha\cdot N),\Lambda(\alpha\cdot N)\}
=-\big(\Lambda(\alpha\cdot N)\big)^2.
\end{equation}
\end{lema}

\begin{proof}
If $\psi$ is a function which is regular in $\Sigma^c$ and $\psi_\pm$ denote the boundary values of $\psi$ on $\Sigma$ when we approach from $\Omega_\pm$, then
\begin{equation}\label{conf l1 eq1}
H(\psi)=\chi_{\Sigma^c} H(\psi)\mu-i(\alpha\cdot N)(\psi_--\psi_+)\sigma
\end{equation}
in the sense of distributions. The proof of this formula, which follows essentially by Stoke's theorem, is an easy exercise left for the reader.

For any $G\in L^2(\mu)^4$, since $H(\Phi(G))=G$ in the sense of distributions, by (\ref{conf l1 eq1}) we have
\begin{equation*}
H(\chi_{\Omega_{\pm}}\Phi(G))=\chi_{\Omega_{\pm}}G\mu\pm i(\alpha\cdot N)\Phi_\sigma(G)\sigma,
\end{equation*}
which implies that 
$\chi_{\Omega_{\pm}}\Phi(G)=\Phi(\chi_{\Omega_{\pm}}G)\pm i\Phi((\alpha\cdot N)\Phi_\sigma(G))$. Recall that 
$\chi_{\Omega_{\pm}}\Phi(G)=0$ in $\Omega_\mp$, so the boundary values $(\chi_{\Omega_{\pm}}\Phi(G))_\mp$ vanish on $\Sigma$.  Combining this fact with Lemma \ref{l jump}$(i)$, we obtain
\begin{equation}\label{conf l1 eq2}
\begin{split}
0&=(\chi_{\Omega_{\pm}}\Phi(G))_\mp
=\Phi_\sigma(\chi_{\Omega_{\pm}}G)\pm iC_\mp((\alpha\cdot N)\Phi_\sigma(G))\\
&=\Phi_\sigma(\chi_{\Omega_{\pm}}G)+\left(-\frac{1}{2}
\pm iC_\sigma(\alpha\cdot N)\right)\Phi_\sigma(G).
\end{split}
\end{equation}

Given $\varphi=\Phi(G+g)\in D(T)$, let $\varphi_{\pm}=\Phi_\sigma(G)+C_{\pm}(g)$ denote the boundary values of $\varphi$ on $\Sigma$.  Since $H(\varphi)=G$ in $\Sigma^c$ and $\Phi_\sigma(G)=\Lambda(g)$, using (\ref{conf l1 eq1}) and Lemma \ref{l jump}$(i)$ we obtain
\begin{equation}\label{conf l1 eq4}
\begin{split}
H(\chi_{\Omega_{\pm}}\varphi)
&=\chi_{\Omega_\pm}G\mu
\pm i(\alpha\cdot N)\big(\Phi_\sigma(G)+C_{\pm}(g)\big)\sigma\\
&=\chi_{\Omega_\pm}G\mu+
\left(\frac{1}{2}\pm i(\alpha\cdot N)(\Lambda+C_\sigma)\right)(g)\sigma
=\chi_{\Omega_\pm}G\mu+f_\pm\sigma,
\end{split}
\end{equation}
where $f_\pm=\big(\frac{1}{2}\pm i(\alpha\cdot N)(\Lambda+C_\sigma)\big)(g)$. This implies that
$\chi_{\Omega_{\pm}}\varphi=\Phi(\chi_{\Omega_\pm}G+f_\pm)$. Hence, $\chi_{\Omega_{\pm}}\varphi\in D(T)$ if and only if 
$\Phi_\sigma(\chi_{\Omega_\pm}G)=\Lambda(f_\pm)$ which, by (\ref{conf l1 eq2}), the definition of $f_\pm$ and that $\Phi_\sigma(G)=\Lambda(g)$, is equivalent to $C_\sigma(\alpha\cdot N)\Lambda(g)=-\Lambda (\alpha\cdot N)(\Lambda+C_\sigma)(g)$. Therefore, 
$\chi_{\Omega_{\pm}}\varphi\in D(T)$ for all $\varphi\in D(T)$ if
$$C_\sigma(\alpha\cdot N)\Lambda+\Lambda (\alpha\cdot N)C_\sigma=-\Lambda (\alpha\cdot N)\Lambda,$$
which is equivalent to $(\ref{conf l1 eq3})$; the lemma is proved. Note that, since $D(T)$ is a vector space, in the statement of the lemma one only needs to require that $\chi_{\Omega_{+}}\varphi\in D(T)$ for all $\varphi\in D(T)$.
\end{proof}

\begin{remark}\label{rem ult}
From the last part of the proof of Lemma \ref{conf l1}, we actually see that $\chi_{\Omega_{\pm}}\varphi\in D(T)$ for all $\varphi\in D(T)$ if and only if $(\ref{conf l1 eq3})$ holds on the set of functions  $g\in L^2(\sigma)^4$ such that there exists $G\in L^2(\mu)^4$ with $\Phi(G+g)\in D(T)$.
\end{remark}

The following lemma is quite standard (see \cite[Section V]{Dittrich}, for example), but we give a proof of it for the sake of completeness.

\begin{lema}\label{conf l2}
Let $T$ be as in Theorem \ref{pre t1} and assume that $T$ is self-adjoint. 
Define the projections $E_\pm:L^2(\mu)^4\to L^2(\mu)^4$ by $E_\pm(\varphi)=\chi_{\Omega_\pm}\varphi$. Then the following are equivalent:
\begin{itemize}
\item[$(i)$] $E_\pm (D(T))\subset D(T)$,
\item[$(ii)$] $E_\pm T\subset T E_\pm$,
\item[$(iii)$] $L_\pm^2(\mu)^4=\{\varphi\in L^2(\mu)^4: \supp(\varphi)\subset\overline{\Omega_\pm}\}$ is invariant under 
$e^{-iTt}$ for all $t\in\R$. 
\end{itemize}
\end{lema}

\begin{proof}
Let us prove $(i)\Longrightarrow(ii)$. Since $E_\pm$ are bounded operators, $D(E_\pm T)=D(T)$. By $(i)$,
$$D(TE_\pm)=\{\varphi\in L^2(\mu)^2: E_\pm(\varphi)\in D(T)\}\supset D(T),$$
thus  $D(E_\pm T)\subset D(TE_\pm)$. If $\varphi=\Phi(G+g)\in D(T)$ then $E_\pm T(\varphi)=E_\pm G=\chi_{\Omega_\pm}G$, but from (\ref{conf l1 eq4}) we have seen that $\chi_{\Omega_{\pm}}\varphi=\Phi(\chi_{\Omega_\pm}G+f_\pm)$ for some $f_\pm\in L^2(\sigma)^4$, so
$$TE_\pm(\varphi)=T(\chi_{\Omega_\pm}\varphi)
=(H+V)\Phi(\chi_{\Omega_\pm}G+f_\pm)=\chi_{\Omega_\pm}G.$$
Therefore, $E_\pm T=TE_\pm$ on $D(T)=D(E_\pm T)$, and $(ii)$ is proved. The implication $(ii)\Longrightarrow(i)$ is straightforward.

In order to prove $(ii)\Longleftrightarrow(iii)$, recall the well known fact that, if $T$ is self-adjoint, then $1+itT:D(T)\to L^2(\mu)^4$ is invertible for all $t\in\R$. This assertion can be easily verified using the arguments in the proof of \cite[Theorem VIII.3]{RS}, for example. By definition, $(ii)$ is equivalent to
\begin{equation*}
E_\pm(1+itT)(\varphi)=(1+itT)E_\pm(\varphi)\quad\text{for all }\varphi\in D(E_\pm T)=D(T)
\end{equation*}
which, by writting $\varphi=(1+itT)^{-1}(\psi)$, is further equivalent to
\begin{equation*}
E_\pm (\psi)=(1+itT)E_\pm(1+itT)^{-1}(\psi)\quad\text{for all }\psi\in L^2(\mu)^4.
\end{equation*}
In conclusion, 
\begin{equation}\label{conf l2 eq1}
(ii)\quad\Longleftrightarrow\quad
(1+itT)^{-1}E_\pm=E_\pm(1+itT)^{-1}\quad\text{for all }t\in\R.
\end{equation}

It is well know that, if $u(t)\in\CC(\R;L^2(\mu)^4)$ is a solution of 
\begin{equation*}
\left\{\begin{split}
&\partial_t u(t)+iTu(t)=0,\\
&u(0)=\psi_0\in L^2(\mu)^4,
\end{split}\right.
\end{equation*}
then one can write $$u(t)=e^{-iTt}(\psi_0)
=\lim_{n\to\infty}\left(\left(1+\textstyle{\frac{it}{n}}T\right)^{-1}\right)^n(\psi_0)$$
(see \cite[Theorem 7.9]{Brezis} for a similar result).
Assume that $\psi_0\in L^2_+(\mu)^4\cup  L^2_-(\mu)^4$, so $E_\pm (\psi_0)=\psi_0$. By (\ref{conf l2 eq1}), we have
\begin{equation*}
E_\pm(u(t))=\lim_{n\to\infty}E_\pm\left(\left(1+\textstyle{\frac{it}{n}}T\right)^{-1}\right)^n(\psi_0)
=\lim_{n\to\infty}\left(\left(1+\textstyle{\frac{it}{n}}T\right)^{-1}\right)^nE_\pm(\psi_0)
=u(t).
\end{equation*}
Therefore, we have proved that if $(ii)$ holds then $E_\pm(u(t))=u(t)$ for all $t\in\R$, which is a restatement of $(iii)$. The implication $(iii)\Longrightarrow(ii)$ is left for the reader.
\end{proof}

\begin{teo}\label{conf t1}
Let $T=H+V$ be as in Theorem \ref{pre t1} and assume that $T$ is self-adjoint. Then, $H+V$ makes $\Sigma$ impenetrable for the particles if $(\ref{conf l1 eq3})$ holds.
\end{teo}
\begin{proof}
That $H+V$ makes $\Sigma$ impenetrable means that the particles under consideration which are initially confined either in $\Omega_+$ or $\Omega_-$ at time $t=0$, remain confined in $\Omega_+$ or in $\Omega_-$ for all $t\in\R$ under the evolution given by $\partial_t=-i(H+V)$, i.e., that Lemma \ref{conf l2}$(iii)$ holds. Thus,
the theorem is a straightforward application of Lemmata \ref{conf l1} and \ref{conf l2}.
\end{proof}

\subsection{Electrostatic and Lorentz scalar shell potentials}
The following theorem is an application of the confinement criterion stated in Theorem \ref{conf t1} to electrostatic and Lorentz scalar shell potentials.
\begin{teo}\label{conf appl}
Assume that $\Sigma$ is $\CC^2$. Given $\lambda_e,\lambda_s\in\R$ such that $|\lambda_e|\neq|\lambda_s|$, let $T$ be the operator defined by
$D(T)=\big\{\Phi(G+g): G\mu+g\sigma\in\XX,\,\Phi_\sigma(G)=\Lambda(g)\big\}$
and $T=H+V_{es}$ on $D(T)$, where 
$$\Lambda=\frac{\lambda_s\beta-\lambda_e}{\lambda_e^2-\lambda_s^2}\,-C_\sigma,\qquad
V_{es}(\varphi)=\frac{1}{2}\,(\lambda_e+\lambda_s\beta)(\varphi_++\varphi_-)\sigma,$$ and 
$\varphi_\pm=\Phi_\sigma(G)+C_\pm (g)$ for $\varphi=\Phi(G+g)\in D(T)$. 
If $\lambda_e^2-\lambda_s^2\neq4$ then $T$ is self-adjoint. In that case, $T$ makes $\Sigma$ impenetrable if and only if $\lambda_e^2-\lambda_s^2=-4$.
\end{teo}

\begin{proof}
If $\varphi=\Phi(G+g)\in D(T)$, using Lemma \ref{l jump}$(i)$ and that $\Phi_\sigma(G)=\Lambda(g)$ we have
\begin{equation}\label{conf t2 eq1}
\begin{split}
V_{es}(\varphi)
&=(\lambda_e+\lambda_s\beta)(\Phi_\sigma(G)+C_\sigma(g))\sigma\\
&=\frac{1}{\lambda_e^2-\lambda_s^2}(\lambda_e+\lambda_s\beta)(\lambda_s\beta-\lambda_e)(g)\sigma=-g\sigma.
\end{split}
\end{equation}
Thus $V_{es}=V$ on $D(T)$, and so $T$ is as in Theorem \ref{pre t1}.

For proving the self-adjointness of $T$ when $\lambda_e^2-\lambda_s^2\neq4$, we follow the arguments of the proof of \cite[Theorem 3.8]{AMV}. Set
$$\Lambda_\pm=\frac{\lambda_e\mp\lambda_s\beta}{\lambda_e^2-\lambda_s^2}\,\pm C_\sigma,$$
so $\Lambda=-\Lambda_+$,
and observe that $\Lambda_\pm$ are self-adjoint on $L^2(\sigma)^4$ because $\beta$ and $C_\sigma$ also are. Using that $(\alpha\cdot N)^2=\beta^2=I_4$ and Lemma \ref{l jump}$(ii)$ we have
\begin{equation}\label{conf t2 eq2}
\begin{split}
\Lambda_+\Lambda_-
&=\frac{\lambda_e^2}{(\lambda_e^2-\lambda_s^2)^2}
-\left(\frac{\lambda_s}{\lambda_e^2-\lambda_s^2}\,\beta-C_\sigma\right)^2
=\frac{1}{\lambda_e^2-\lambda_s^2}
+\frac{\lambda_s}{\lambda_e^2-\lambda_s^2}\,\{\beta,C_\sigma\}-C_\sigma^2\\
&=\frac{1}{\lambda_e^2-\lambda_s^2}-\frac{1}{4}
+\frac{\lambda_s}{\lambda_e^2-\lambda_s^2}\,\{\beta,C_\sigma\}-C_\sigma(\alpha\cdot N)\{\alpha\cdot N,C_\sigma\}
=b-K,
\end{split}
\end{equation}
where $b=1/(\lambda_e^2-\lambda_s^2)-1/4$ and 
$K=(\lambda_s^2-\lambda_e^2)^{-1}\lambda_s\{\beta,C_\sigma\}+C_\sigma(\alpha\cdot N)\{\alpha\cdot N,C_\sigma\}$.
In \cite[Lemma 3.5]{AMV} we proved that $\{\alpha\cdot N,C_\sigma\}$ is a compact operator in $L^2(\sigma)^4$, and since $\beta$ anticommutes with the $\alpha_j$'s, we easily have
\begin{equation*}
\{\beta,C_\sigma\}g(x)=\frac{m}{2\pi}\int_\Sigma\frac{e^{-m|x-z|}}{|x-z|}\,g(z)\,d\sigma(z).
\end{equation*}
Thus $\{\beta,C_\sigma\}$ is a compact operator by \cite[Proposition 3.11]{Folland} and hence $K$ is also compact. 

If $\lambda_e^2-\lambda_s^2\neq4$ then $b\neq0$ and Fredholm's theorem applies to $b-K$ (see, \cite[Theorem 0.38$(c)$]{Folland}). If for example $\lambda_e\neq0$, using (\ref{conf t2 eq2}) we can follow the proof of \cite[Lemma 3.7]{AMV} to show that $\Lambda$ has closed range. Moreover, as we did in the first part of the proof of \cite[Theorem 3.8]{AMV}, Fredholm's theorem also shows that $\{\Phi(h): h\in\Ker(\Lambda)\}$ is closed, we omit the details.

In any case, that $\Ran(\Lambda)$ and $\{{\Phi(h)}:\,h\in\Ker(\Lambda)\}$ are closed for all $\lambda_e^2-\lambda_s^2\neq4$ follows by (\ref{conf t2 eq2}), Fredholm's theorem and \cite[Theorem 1.46$(ii)$]{Aiena}, so the restriction $\lambda_e\neq0$ is not necessary. 
These properties of $\Lambda$ together with (\ref{conf t2 eq1}) allow us to apply Theorem \ref{pre t1}, which proves that $T$ is self-adjoint for $\lambda_e^2-\lambda_s^2\neq4$. 

Let us finally check the impenetrability condition relative to $T$ and $\Sigma$. By Theorem \ref{conf t1}, $T$ makes $\Sigma$ impenetrable for the particles if $\Lambda$ satisfies $(\ref{conf l1 eq3})$. By a straightforward computation using Lemma \ref{l jump}$(ii)$, that $(\alpha\cdot N)^2=\beta^2=I_4$, and that $(\alpha\cdot N)$ and $\beta$ anticommute, we have
\begin{equation}\label{kpopkosa}
\begin{split}
\{C_\sigma(\alpha\cdot N),\Lambda(\alpha\cdot N)\}
+\big(\Lambda(\alpha\cdot N)\big)^2
=\frac{1}{4}+\frac{1}{\lambda_e^2-\lambda_s^2}.
\end{split}
\end{equation}
Therefore, $\Lambda$ satisfies $(\ref{conf l1 eq3})$ if and only if $\lambda_e^2-\lambda_s^2=-4$, and in such case $\Sigma$ becomes impenetrable. Furthermore, since the right hand side of (\ref{kpopkosa}) is a constant times $I_4$, from Remark \ref{rem ult} and Lemma \ref{conf l2} we actually deduce that $T$ makes $\Sigma$ impenetrable if and only if $\lambda_e^2-\lambda_s^2=-4$.
\end{proof}

We have seen that $H+V_{es}$ makes $\Sigma$ impenetrable if and only if $\lambda_e^2-\lambda_s^2+4=0$, which is precisely equation $(5.1)$ of \cite[Section V]{Dittrich}. Hence, for the potentials $V_{es}$, our results on confinement generalize the ones stated in \cite{Dittrich}  to regular surfaces.

\end{document}